\documentclass[12pt]{amsart}
\usepackage{pifont}
\usepackage{txfonts}
\usepackage{}
\usepackage{amsfonts}
\usepackage{graphicx}
\usepackage{epstopdf}
\usepackage{amsmath}
\usepackage{amsthm}
\usepackage{latexsym,bm}
\usepackage{amssymb}
\usepackage{esint}
\usepackage{enumitem}
\usepackage{indentfirst}
\usepackage{amscd}
\newtheorem{thm}{Theorem}[section]

\newtheorem{lem}[thm]{Lemma}

\newtheorem{defn}[thm]{Definition}

\newtheorem{Remark}{Remark}
\numberwithin{equation}{section}
\numberwithin{Remark}{section}

\begin{document}

\title{On Uniqueness And Existence of Conformally Compact Einstein Metrics with Homogeneous Conformal Infinity}

\author{Gang Li$^\dag$}

\begin{abstract} In this paper we show that for a generalized Berger metric $\hat{g}$ on $S^3$ close to the round metric, the conformally compact Einstein (CCE) manifold $(M, g)$ with $(S^3, [\hat{g}])$ as its conformal infinity is unique up to isometries. For the high-dimensional case, we show that if $\hat{g}$ is an $\text{SU}(k+1)$-invariant metric on $S^{2k+1}$ for $k\geq1$, the non-positively curved CCE metric on the $(2k+1)$-ball $B_1(0)$ with $(S^{2k+1}, [\hat{g}])$ as its conformal infinity is unique up to isometries. In particular, since in \cite{LiQingShi}, we proved that if the Yamabe constant of the conformal infinity $Y(S^{2k+1}, [\hat{g}])$ is close to that of the round sphere then any CCE manifold filled in must be negatively curved and simply connected, therefore if $\hat{g}$ is an $\text{SU}(k+1)$-invariant metric on $S^{2k+1}$ which is close to the round metric, the CCE metric filled in is unique up to isometries. Using the continuity method, we prove an existence result of the non-positively curved CCE metric with prescribed conformal infinity $(S^{2k+1}, [\hat{g}])$ when the metric $\hat{g}$ is $\text{SU}(k+1)$-invariant.
\end{abstract}

\renewcommand{\subjclassname}{\textup{2000} Mathematics Subject Classification}


\thanks{$^\dag$ Research partially supported by the National Natural Science Foundation of China No. 11701326 and the Fundamental Research Funds of Shandong University 2016HW008.}

\address{Gang Li, Department of Mathematics, Shandong University, Jinan, Shandong Province, China}
\email{runxing3@gmail.com}

\maketitle


\section{Introduction}

This is a continuation of our previous work (\cite{Li}) on uniqueness of conformally compact Einstein (CCE) metrics (see Definition \ref{defn_conformallycompactEinsteinmetric}) with prescribed homogeneous conformal infinity. Let $B_1$ be the unit ball in the Euclidean space $\mathbb{R}^{n+1}$ of dimension $(n+1)$, with its boundary $\mathbb{S}^n$ the unit sphere. For a given homogeneous metric $\hat{g}$ on $\mathbb{S}^n$, in this paper we mainly focus on the uniqueness and existence of non-positively curved CCE metrics $g$ on $B_1$ with $(\mathbb{S}^n, [\hat{g}])$ as its conformal infinity.

In \cite{GL}, for a Riemannian metric $\hat{g}$ on the $n$-sphere $\mathbb{S}^n$ which is $C^{2,\alpha}$ close to the round metric, Graham and Lee proved the seminal existence result that there exists a CCE metric on the $(n+1)$-ball $B_1(0)$ with $(\mathbb{S}^n, [\hat{g}])$ as its conformal infinity, and the solution is unique in a small neighborhood of the asymptotic solution they constructed in a weighted space by the implicit function theorem. Later Lee generalized this perturbation result to more general CCE manifolds which include the case when they are non-positively curved. It is interesting to understand whether the solution is globally unique with the prescribed conformal infinity. On the other hand, in light of LeBrun's local construction in \cite{LeBrun1}, when the conformal infinity is a Berger metric on $\mathbb{S}^3$ or a generalized Berger metric which is left invariant under the $\text{SU}(2)$ action, Pedersen \cite{Pedersen} and Hitchin \cite{Hitchin1} could fill in a global CCE metric on the $4$-ball, which has self-dual Weyl curvature, and the metric is unique under the self-duality assumption. If both the conformal infinity $(\mathbb{S}^n, [\hat{g}])$ and the non-local term in the expansion (see Theorem \ref{thm_expansion1}) of the Einstein metric at infinity are given, Anderson \cite{Anderson} and Biquard \cite{Biquard2} proved that the CCE metric is unique up to isometry. When $\hat{g}$ is the round metric, it is proved that the CCE metric filled in must be the hyperbolic space, see \cite{Andersson-Dahl}\cite{Q}\cite{DJ}\cite{LiQingShi}, see also \cite{CLW}. In general, given the conformal infinity, the CCE metrics filled in is not necessarily unique, see \cite{HP} and \cite{Anderson}. For instance, there are CCE metrics on $\mathbb{S}^2\times \mathbb{R}^2$ which are not isometric, with the same conformal infinity on $\mathbb{S}^2\times\mathbb{S}^1$.

Based on \cite{ST}, \cite{WY} and \cite{DJ}, in \cite{LiQingShi} we proved that given a conformal infinity $(\mathbb{S}^n, [\hat{g}])$ with its Yamabe constant close to that of the round sphere metric, the CCE manifold filled in must be a Hadamard
 manifold (a simply connected non-positively curved complete non-compact Riemannian manifold), with its sectional curvature close to $-1$ uniformly. So it is natural to consider the uniqueness problem under the assumption that the CCE metric is non-positively curved. Let $(M^{n+1}, g)$ be a non-positively curved CCE with $(\partial M, [\hat{g}])$ as its conformal infinity. In \cite{Li}, inspired by X. Wang's work in \cite{Wang} (see also \cite{Andersson-Dahl} and \cite{Anderson1}) we were able to show that any smooth conformal Killing vector field $Y$ on $(\partial M, [\hat{g}])$ can extend continuously to a Killing vector field $X$ on $(M, g)$. As an application, if we also assume that $M$ is simply connected and the conformal infinity is $(\mathbb{S}^n, [\hat{g}])$ with $\hat{g}$ some homogeneous metric on $\mathbb{S}^n$, we proved that $g$ is homogeneous when restricted to any geodesic spheres centered at the unique fixed point $p_0\in M$ (called {\it the center of gravity}) of the group action, which is generated by the Killing vector fields extended from the conformal infinity, and we have a natural geodesic defining function
 \begin{align*}
 x=e^{-r}
 \end{align*}
 about $C\hat{g}$, with $r$ the distance function to $p_0$ on $(M, g)$ and some constant $C>0$; moreover, when the conformal infinity is $(\mathbb{S}^3, [\hat{g}])$ with $\hat{g}$ the Berger metric or the generalized Berger metric, along a geodesic line $\gamma$ connecting $p_0$ and a given point $q\in \partial M$,  we deformed the problem to a boundary value problem of a system of  second order elliptic type ODEs (of  three functions $(y_1(x), \,y_2(x),\,y_3(x))$) on $x\in[0,1]$,  which are degenerate at the two boundary points, see $(\ref{equn_GBergerEinstein01})-(\ref{equn_GBergerBV01})$. This approach gives a natural way to fix the gauge of the Einstein equations for this special problem. For the Berger metric case, either $y_2$ or $y_3$ vanishes identically on $x\in[0,1]$ and we were able to use the monotonicity of the solutions $y_i(x)$ on $x\in[0,1]$ and an integral version comparison theorem to show the uniqueness of the solution to be boundary value problem i.e., when $\hat{g}$ is a Berger metric on $\mathbb{S}^3$, up to isometry there exists at most one CCE metric on $B_1$ which is non-positively curved with $(\mathbb{S}^3, [\hat{g}])$ its conformal infinity, and hence it is the metric constructed by Pedersen in \cite{Pedersen}. Moreover, by \cite{LiQingShi}, for the Berger metric $\hat{g}$ in $(\mathbb{S}^3, [\hat{g}])$ close enough to the round sphere metric, the CCE manifold filled in is automatically simply connected and non-positively curved, and therefore it is unique up to isometry, which is Pedersen's metric in \cite{Pedersen} and also Graham-Lee's metric in \cite{GL}.

For the generalized Berger metric $\hat{g}$ on $\mathbb{S}^3$, by the interaction of 
$y_1(x),\,y_2(x)$ and $\,y_3(x)$ in the system $(\ref{equn_GBergerEinstein01})-(\ref{equn_GBergerEinstein04})$, the 
integral comparison argument fails to give uniqueness of the solution to this boundary value problem. In this paper, we consider the uniqueness of the solution in this case by a contradiction argument. Based on a monotonicity argument and an a priori estimate of the solution, we consider the total variation of the difference of two solutions and show that
\begin{thm}\label{thm_Einsteineqnsbvpuniqueness1}
 Let $(\mathbb{S}^3, \hat{g})$ be a generalized Berger sphere with the standard form
 \begin{align}\label{equn_Bergermetric1}
\hat{g}=\lambda_1 \sigma_1^2\,+\,\lambda_2d\sigma_2^2\,+\,\lambda_3 d\sigma_3^2,
\end{align}
with $\sigma_1, \,\sigma_2$ and $\sigma_3$ three $\text{SU}(2)$-invariant $1$-forms, such that $\lambda_1, \lambda_2$ and $ \lambda_3$ differ from one another and without loss of generality we assume $\lambda_1>\max\{\lambda_2, \,\lambda_3\}$. Assume also that $\phi_1(0)=\frac{\lambda_2}{\lambda_1}$ and $\phi_2(0)=\frac{\lambda_3}{\lambda_2}$ satisfy the inequality $1<\phi_1(0)+\phi_1(0)\phi_2(0)$. Then there exists some constant 
 $0< \eta_0 <1-3\times10^{-8}$, such that if $\eta_0<\phi_1(0)<1$ and $\eta_0 < \phi_2(0)$, then the Hadamard manifold $(M^4, g)$ which is conformally compact Einstein, 
 with $(\mathbb{S}^3, [\hat{g}])$ as its conformal infinity, must be unique up to isometry. And hence, it is Hitchin's metric in \cite{Hitchin1}, and for $\eta_0$ close enough to $1$ it is also Graham-Lee's metric in \cite{GL}.
\end{thm}
The estimates of the solution are based on the monotonicity of the solutions and the choice of different types of integrating factors for different equations, on different parts of the interval $x\in[0,1]$, due to the degeneracy of the elliptic equations. We should remark that the upper bound of the norm of Weyl tensor $|W|_g$ is used in the estimate of the solutions near $x=1$. It is interesting that for the difference of any given two solutions $(z_1,z_2,z_3)=(y_{11}, y_{12}, y_{13})-(y_{21}, y_{22}, y_{31})$, the total variation of $z_2$ and $z_3$ are well controlled by a certain type of integration of the equation only on some special intervals of monotonicity on $x\in[0,1]$, which is enough for our argument. If we could have an estimate that $\sup_M|W|_g$ is small for the given conformal infinity, then the lower bound $\eta_0>0$ in Theorem \ref{thm_Einsteineqnsbvpuniqueness1} could be much smaller, by the estimates in the proof. On the other hand, for $ 0 < \eta_0 < 1$ close enough to $1$, $g$ must be non-positively curved and the quantity $\sup_M|W|_g$ is small enough by Theorem \ref{EHBoundary} (see also \cite{LiQingShi}), and hence as a corollary of Theorem \ref{thm_Einsteineqnsbvpuniqueness1}, we have
 \begin{thm}\label{thm_Einsteineqnsbvpuniqueness2}
 Let $(\mathbb{S}^3, \hat{g})$ be a generalized Berger sphere with the expression $(\ref{equn_Bergermetric1})$ such that $\lambda_1, \lambda_2$ and $ \lambda_3$ differ from one another and $\lambda_1>\max\{\lambda_2, \,\lambda_3\}$. Assume that $\phi_1(0)=\frac{\lambda_2}{\lambda_1}$ and $\phi_2(0)=\frac{\lambda_3}{\lambda_2}$. Then there exists some constant $0< \eta_0 <1$ close enough to $1$, such that if $\eta_0<\phi_1(0)<1$  and $\eta_0 < \phi_2(0)$, then the conformally compact Einstein manifold $(M^4, g)$ with $(\mathbb{S}^3, [\hat{g}])$ as its conformal infinity is unique up to isometry, and hence it must be the (anti-)self-dual metric constructed in \cite{Hitchin1} and also the perturbation metric in \cite{GL}.
\end{thm}
The uniqueness result can be generalized to high dimension. Homogeneous spaces on $\mathbb{S}^n$ have been classified by D. Montgomery and H. Samelson (\cite{MS}), and A. Borel (\cite{Borel1}, \cite{Borel2}), see \cite{Besse} (p.179) and also \cite{Ziller1} for instance. Up to a scaling factor and isometry, homogeneous metrics on spheres are in one of the three classes: a one parameter family of $\text{SU}(k+1)$-invariant metrics on $\mathbb{S}^{2k+1}\cong \text{SU}(k+1)/ \text{SU}(k)$ $(k\geq1)$, a three parameter family of $\text{Sp}(k+1)$-invariant metrics on $\mathbb{S}^{4k+3}\cong \text{Sp}(k+1)/\text{Sp(k)}$ (containing the $\text{SU}(2k+1)$-invariant metrics as a subset in these dimensions) with $k\geq0$, and a one parameter family of $\text{Spin}(9)$-invariant metrics on $\mathbb{S}^{15}\cong \text{Spin}(9)/ \text{Spin}(7)$. For more details, see Section \ref{section4}, and also \cite{Ziller1}. As the Berger metric case, for the first two classes of homogeneous metric $\hat{g}$, the prescribed conformal infinity problem of the CCE metrics is deformed into a two-point boundary value problem of a system of ODEs on the interval $x\in[0,1]$, see $(\ref{equn_SUnEinstein01})-(\ref{equn_SUnBV01})$ (with two functions $(y_1(x),\,y_2(x))$) and $(\ref{equn_SpnEinstein01})-(\ref{equn_SpnBV01})$ (with four functions $(y_1(x),..,y_4(x))$) correspondingly. The third class could be done similarly. Using the same approach of the uniqueness argument for the Berger metric case in \cite{Li} based on a monotonicity argument and an integral version comparison theorem, we show that
\begin{thm}\label{thm_someSUnmetric}
Let $\hat{g}$ be a homogeneous metric on $\mathbb{S}^n\cong \text{SU}(k+1)/\text{SU}(k)$ with $n=2k+1$ for $k\geq 1$ so that $\hat{g}$ has the standard diagonal form
\begin{align}\label{eqn_SUnstandardmetricform}
\hat{g}=\lambda_1\sigma_1^2+\lambda_2(\sigma_2^2+..+\sigma_n^2),
\end{align}
at a point where $\lambda_1$ and $\lambda_2$ are two positive constants and $\sigma_1,..,\sigma_n$ are the $1$-forms with respect to the basis vectors in $\mathfrak{p}$, in the $\text{Ad}_{SU(k)}$-invariant splitting $su(k+1)=su(k)\oplus \mathfrak{p}$. Assume that $\frac{1}{n+1}<\frac{\lambda_1}{\lambda_2}<n+1$, then up to isometry there exists at most one non-positively curved conformally compact Einstein metric on the $(n+1)$-ball $B_1(0)$ with $(\mathbb{S}^n, [\hat{g}])$ as its conformal infinity. In particular, it is the perturbation metric in \cite{GL} when $\frac{\lambda_1}{\lambda_2}$ is close to $1$. Moreover, by Theorem \ref{EHBoundary} (see also \cite{LiQingShi}), for $\frac{\lambda_1}{\lambda_2}$ close enough to $1$, any conformally compact Einstein manifold filled in is automatically negatively curved and simply connected, and therefore it is unique up to isometry.
\end{thm}
After that we consider the existence result. Recall that for given real analytic data i.e., the conformal infinity $(\partial M, [\hat{g}])$ and the non-local term in the expansion in Theorem \ref{thm_expansion1}, existence of CCE metrics in a neighborhood of the boundary $\partial M$ is proved by Fefferman and Graham in \cite{FG} for $\partial M$ of odd dimension and by Kichenassamy in \cite{SKichenassamy} for even dimensional boundary, and for $C^{\infty}$ data at conformal infinity, see Gursky and Sz$\acute{\text{e}}$kelyhidi \cite{GS}. Anderson \cite{Anderson1} studied the existence of CCE metrics on $B_1^4$ with general prescribed conformal infinity $(\mathbb{S}^3, [\hat{g}])$ using the continuity method. Recently, Gursky and Han \cite{GH} showed that there are infinitely many Riemannian metrics $\hat{g}$ on $\mathbb{S}^7$ lying in different connected components of the set of positive scalar curvature metrics such that there exists no CCE metrics on the unit Euclidean ball $B^8$ with $(\mathbb{S}^7, [\hat{g}])$ as its conformal infinity and pointed out that similar phenomena holds for higher dimensions.

By \cite{Li}, the CCE manifold which is Hadamard with homogeneous conformal infinity, is {\it of cohomogeneity one}. Calculations of the curvature tensors on manifolds of cohomogeneity one can be found in \cite{GZ}. Recently, using Schauder degree theory, Buttsworth \cite{TB} showed that for two $G$-invariant Riemannian metrics $\hat{g}_1$ and $\hat{g}_2$ on a compact homogeneous space $G/H$, if the isotropy representation of $G/H$ consists of pairwise inequivalent irreducible
summands, then there exists an Einstein metric on $G/H \times [0,1]$ such that when restricted on $G/H \times\{0\}$ and $G/H\times \{1\}$, $g$ coincides with $\hat{g}_1$ and $\hat{g}_2$ respectively.

We prove a compactness result of a sequence of non-positively curved CCE metrics with their conformal infinity $(\mathbb{S}^n, [\hat{g}_j])$ $(j\geq 1)$, where $\hat{g}_j$ is of the form $(\ref{eqn_SUnstandardmetricform})$ and uniformly bounded. Using this compactness result and Graham-Lee and Lee's perturbation result in \cite{GL} and \cite{Lee}, we show the following existence theorem by the continuity method.

\begin{thm}\label{thm_someSUnmetric1}
Let $B_1\subseteq \mathbb{R}^{n+1}$ be the unit ball on the Euclidean space with the unit sphere $\mathbb{S}^n$ as its boundary. Assume that $n=2k+1$ for some integer $k\geq 1$. Let $\hat{g}^{\lambda}$ be a homogeneous metric on the boundary $\mathbb{S}^n\cong \text{SU}(k+1)/\text{SU}(k)$ 
so that $\hat{g}$ has the standard diagonal form
\begin{align*}
\hat{g}^{\lambda}=\sigma_1^2+\lambda(\sigma_2^2+..+\sigma_n^2),
\end{align*}
 at a point where $\lambda$ is a positive constant and $\sigma_1,..,\sigma_n$ are the $1$-forms with respect to the basis vectors in $\mathfrak{p}$ in the $\text{Ad}_{SU(k)}$-invariant splitting $su(k+1)=su(k)\oplus \mathfrak{p}$. Then as the parameter $\lambda$ varies from $\lambda=1$ continuously on the interval $(\frac{1}{(n+1)}, 1]$ (resp. on the interval $[1, n+1)$), either it holds that there exists a conformally compact Einstein metric on $B_1$ which is non-positively curved with $(\mathbb{S}^n, [\hat{g}^{\lambda}])$ as its conformal infinity for each $\lambda\in (\frac{1}{n+1}, 1]$ (resp. $\lambda\in [1, n+1)$); or there exists $\lambda_1\in(\frac{1}{n+1}, 1]$ (resp. $\lambda_1\in[1, n+1)$), such that for each $\lambda\in [\lambda_1, 1]$ (resp. $\lambda\in[1,\lambda_1]$) there exists a conformally compact Einstein metric $g^{\lambda}$ on $B_1$ which is non-positively curved with $(\mathbb{S}^n, [\hat{g}^{\lambda}])$ as its conformal infinity and there exists $p\in B_1$ such that the sectional curvature of $g^{\lambda_1}$ is zero in some direction at $p$ and moreover, for any $\epsilon>0$ small there exists $\lambda_2\in(\lambda_1-\epsilon, \lambda_1)$ (resp. $\lambda_2\in(\lambda_1, \lambda_1+\epsilon)$) such that there exists a conformally compact Einstein metric $g^{\lambda_2}$ on $B_1$ with $(\mathbb{S}^n, [\hat{g}^{\lambda_2}])$ as its conformal infinity and the sectional curvature of $g^{\lambda_2}$ is positive in some direction at some point $p\in B_1$.
\end{thm}
This existence result can be viewed as a generalization of Pedersen's result in \cite{Pedersen} to higher dimensions. For $n=3$, the constant $\lambda_1$ in Theorem \ref{thm_someSUnmetric1} can be calculated explicitly by Pedersen's explicit solutions. Also, a similar existence result holds for the generalized Berger metric case (which should be the Hitchin's metric for $\hat{g}$ close to the round sphere by uniqueness), since we have the same compactness result for it. Notice that the uniqueness result and existence result should hold for the conformal infinity $(\mathbb{S}^{15}, [\hat{g}])$ with $\hat{g}$ a $\text{Spin}(9)$-invariant metric by the same argument, due to the fact that in the system of ODEs obtained there are only two unknown functions.

The uniqueness and existence of the conformally compact Einstein metric with prescribed
conformal infinity $(\mathbb{S}^{4k+3}, [\hat{g}])$ where $\hat{g}$ is $\text{Sp}(k+1)$-invariant will be discussed else where.

In Section \ref{section3}, we first show that each function $y_i(x)$ in the solution $(y_1,y_2,y_3)$ to the boundary value problem $(\ref{equn_GBergerEinstein01})-(\ref{equn_GBergerBV01})$ is monotone on $x\in[0,1]$, see Lemma \ref{lem_monotonicity301}. Based on this and the elliptic system, we obtain the interior estimate of the solution in $x\in(0,1)$. Then we employ certain integrating factors to deal with the degeneracy of the equations at $x=0$ and obtain an a priori estimate of the solutions on $x\in[0, \frac{3}{4}]$, see Lemma \ref{lem_uniformests301}. That fails to work near $x=1$. We have to use the boundedness of the Weyl tensor for the non-positively curved Einstein metric and the Einstein equation to give an estimate of the solution near $x=1$, 
see Lemma \ref{lem_yiboundrightGB}. Then we go to the proof of Theorem \ref{thm_Einsteineqnsbvpuniqueness1} by a contradiction argument. Assume we have two solutions $(y_{11}, y_{12}, y_{13})$ and $(y_{21}, y_{22}, y_{31})$, we consider the total variation of each function $z_i$ in the difference $(z_1,z_2,z_3)$ of the two solutions. By employing certain integrating factors, we show that on some "good" intervals of monotonicity of $z_2$ ( resp. $z_3$), the total variation of $z_2$ (resp. $z_3$) is controlled by the summation of those of $z_1$ and $z_3$ (resp. $z_2$) with relatively small coefficients, and the "smallness" of the coefficients is due to the "smallness" of $(1-\varepsilon_0)\sup_M|W|_g$ by the choice of $\varepsilon_0\in(0,1)$ and the closeness of the initial data to that of the round metric in the estimates of the solutions on $x\in[0,\varepsilon_0]$. Controls on the total variation of $z_1$ (by the summation of those of $z_2$ and $z_3$ with small coefficients) holds on each interval of monotonicity of $z_1$ by a different integral argument. Therefore, if either we have the estimate that the normal of the Weyl tensor is small, or we just choose $\varepsilon_0$ close to $1$ due to the uniform bound of the Weyl tensor, we can choose $0<\varepsilon_0<1$ so that the total variation of $z_i$ is well controlled on $x\in[\varepsilon_0, 1]$ by those of the other two, and it is also true on $x\in[0,\varepsilon_0]$ by the estimates based on the initial data, provided that the initial data is not quite far from that with respect to the round metric. Then we obtain that the total variation of each $z_i$ vanishes on $x\in[0,1]$, which proves the uniqueness result in Theorem \ref{thm_Einsteineqnsbvpuniqueness1}.

In Section \ref{section4}, for the conformal infinity $(\mathbb{S}^n, [\hat{g}])$ of general dimensions, with $\hat{g}$ a homogeneous metric, we use the symmetry extension to reduce the Einstein equations with the prescribed conformal infinity to a two-point boundary value problem of a system of ODEs on $x\in[0,1]$, see $(\ref{equn_SUnEinstein01})-(\ref{equn_SUnBV01})$ when $\hat{g}$ is $\text{SU}(k+1)$-invariant and $(\ref{equn_SpnEinstein01})-(\ref{equn_SpnBV01})$ when $\hat{g}$ is $\text{Sp}(k+1)$-invariant. Then by the same argument as the Berger metric case in \cite{Li}, we prove Theorem \ref{thm_someSUnmetric}. Based on the monotonicity of the solutions proved in Lemma \ref{lem_monotonicity01}, we give a uniform $C^3$ estimate of the solution to the boundary value problem $(\ref{equn_SUnEinstein01})-(\ref{equn_SUnBV01})$ on $x\in[0, \frac{3}{4}]$ in Lemma \ref{lem_uniformestsn01}. Then combining with the interior estimate of the Einstein metric based on the uniform bound of the Weyl tensor, we get a compactness result of the non-positively curved conformally compact Einstein metrics with bounded conformal infinity. And hence with the aid of the perturbation result in \cite{GL} and \cite{Lee}, we use the continuity method to prove the existence result in Theorem \ref{thm_someSUnmetric1}. Notice that the global compactness estimate on CCE metrics is difficult to obtain in general, see \cite{Anderson1}.

\vskip0.2cm
{\bf Acknowledgements.} The author would like to thank Professor Jie Qing and Professor Yuguang Shi for helpful discussion and constant support.  The author is grateful to Professor Fuquan Fang and Xiaoyang Chen for helpful discussion on homogeneous spaces.

\section{Preliminaries}\label{Sect:preliminary}

\begin{defn}\label{defn_conformallycompactEinsteinmetric}
Suppose $M$ is the interior of a smooth compact manifold $\overline{M}$ of dimension $n+1$ with boundary $\partial M$. A defining function $x$ on $\overline{M}$ is a smooth function $x$ on $\overline{M}$ such that
\begin{align*}
x>0\,\,\text{ in}\,\, M,\,\, x=0\,\,\text{and}\,\,dx\neq 0\,\, \text{on}\,\,\partial M.
\end{align*}
A complete Riemannian metric $g$ on $M$ is said to be conformally compact if there exists a defining function $x$ such that $x^2g$ extends by continuity to a Riemannian metric (of class at least $C^0$) on $\overline{M}$. The rescaled metric $\bar{g}=x^2g$ is called a conformal compactification of $g$. If for some smooth defining function $x$, $\bar{g}$ is in $C^k(\overline{M})$ or the Holder space $C^{k,\alpha}(\overline{M})$, we say $g$ is conformally compact of class $C^k$ or $C^{k, \alpha}$. Moreover, if $g$ is also Einstein, we call $g$ a conformally compact Einstein (CCE) metric. Also, for the restricted metric $\hat{g}=\bar{g}\big|_{\partial M}$, the conformal class $(\partial M, [\hat{g}])$ is called the conformal infinity of $(M, g)$. A defining function $x$ is called a geodesic defining function about $\hat{g}$ if $\hat{g}=\bar{g}\big|_{\partial M}$ and $|dx|_{\bar{g}}=1$ in a neighborhood of the boundary.
\end{defn}
One can easily check that a CCE metric $g$ on $M$ satisfies
\begin{align}\label{eqn_Einstein}
Ric_g=-n g.
\end{align}

Let $(M^{n+1}, g)$ be a Hadamard manifold, 
and we also assume that $(M, g)$ is a CCE manifold with its conformal infinity $(\partial M, [\hat{g}])$. By the non-positivity of the sectional curvature of $g$ and $(\ref{eqn_Einstein})$, we have $-n\leq K\leq 0$ for any sectional curvature $K$ and $|W|_g\leq \sqrt{(n^2-1)n}$ at any point $p\in M$. It is shown in \cite{Li} that if $(M, g)$ is not the hyperbolic space, there exists a unique point $p_0\in M$ which is the center of the unique closed geodesic
ball of the smallest radius that contains the set $S\equiv\{p\in M \big| |W|_g(p)=\sup_M|W|_g\}$. We call $p_0$ the {\it (spherical) center of
gravity of } $(M, g)$. Each conformal Killing vector field $Y$ on $(\partial M, [\hat{g}])$ can be extended continuously to a Killing vector field $X$ on $(M,g)$, with each geodesic sphere centered at $p_0$ an invariant subset of the action generated by $X$ (see \cite{Li}). Let $r$ be the distance function to $p_0$ on $(M, g)$. Then $g$ has the orthogonal splitting
\begin{align}\label{eqn_metrictwocomponents}
g=dr^2+g_r,
\end{align}
with $g_r$ the restriction of $g$ on the $r$-geodesic sphere centered at $p_0$. If moreover, $(\partial M, \hat{g})$ is a homogeneous space, then the function $x=e^{-r}$ is a geodesic defining function about $C\hat{g}$ with $C>0$ some constant. That is to say,  $\hat{g}=C\displaystyle\lim_{x\to 0}(x^2g_r)$. We also have the form
\begin{align}\label{equn_splittingmetric01}
g=dr^2+\sinh^2(r)\bar{h}=x^{-2}(dx^2+\frac{(1-x^2)^2}{4}\bar{h}),
\end{align}
for $0\leq x \leq 1$. Let $(r, \theta)$ be the polar coordinates centered at $p_0$.

Let $\hat{g}$ be a generalized Berger metric on $\mathbb{S}^3$ of the standard form
\begin{align*}
\hat{g}=\lambda_1 \sigma_1^2\,+\,\lambda_2d\sigma_2^2\,+\,\lambda_3 d\sigma_3^2,
\end{align*}
with $\sigma_1, \,\sigma_2$ and $\sigma_3$ three $\text{SU}(2)$-invariant $1$-forms, where $\lambda_1, \lambda_2$ and $ \lambda_3$ differ from one another. It is shown in \cite{Li} that the metrics $\bar{h}$ in $(\ref{equn_splittingmetric01})$ on the geodesic spheres have the diagonal form
\begin{align*}
\bar{h}=I_1(x)d(\theta^1)^2+I_2(x)d(\theta^2)^2+I_3(x)d(\theta^3)^2,
\end{align*}
at a point $(x, \theta_0)$ under the coordinates $(x,\theta)=(x, \theta^1,\theta^2,\theta^3)$ such that $d\theta^i=\sigma_i$ (with $\sigma_i$ an $\text{SU}(2)$-invariant $1$-form) at $\theta=\theta_0$ for $0 \leq x \leq 1$, with some positive functions $I_i\in C^{\infty}([0,1])$ satisfying $I_i(1)=1$, $i=1,2,3$.

Denote $K=I_1I_2I_3$, $\phi_1=\frac{ I_2}{I_1}$, $\phi_2=\frac{I_3}{I_2}$, $y_1=\log(K)$, $y_2=\log(\phi_1)$ and $y_3=\log(\phi_2)$, so that
\begin{align*}
I_1=(K\phi_1^{-2}\phi_2^{-1})^{\frac{1}{3}},\,\,I_2=(K\phi_1\phi_2^{-1})^{\frac{1}{3}},\,\,I_3=(K\phi_1\phi_2^2)^{\frac{1}{3}}.
\end{align*}

It was shown in \cite{Li} that for a generalized Berger metric $\hat{g}$ on $\mathbb{S}^3$ with $\lambda_1, \lambda_2, \lambda_3$ different from one another, the Einstein equations $(\ref{eqn_Einstein})$ with the prescribed conformal infinity $(S^3, [\hat{g}])$ is equivalent to
\begin{align}
&\label{equn_GBergerEinstein01}y_1''-x^{-1}(1+3x^2)(1-x^2)^{-1}y_1'+\frac{1}{6}(y_1')^2+\frac{1}{3}[(y_2')^2+y_2'y_3'+(y_3')^2]=0,\\
&\label{equn_GBergerEinstein02}y_1''-x^{-1}(5+7x^2)(1-x^2)^{-1}y_1'+\frac{1}{2}(y_1')^2+16(1-x^2)^{-2}[3-2K^{-\frac{1}{3}}(\phi_1^2\phi_2)^{\frac{1}{3}}-2K^{-\frac{1}{3}}(\phi_1^{-1}\phi_2)^{\frac{1}{3}}\\
&-2K^{-\frac{1}{3}}(\phi_1\phi_2^2)^{-\frac{1}{3}}+K^{-\frac{1}{3}}\phi_1^{-\frac{4}{3}}\phi_2^{-\frac{2}{3}}+K^{-\frac{1}{3}}\phi_1^{\frac{2}{3}}\phi_2^{-\frac{2}{3}}+K^{-\frac{1}{3}}\phi_1^{\frac{2}{3}}\phi_2^{\frac{4}{3}}]=0,
\notag \\
&\label{equn_GBergerEinstein03}y_2''-2x^{-1}(1+2x^2)(1-x^2)^{-1}y_2'+\frac{1}{2}y_1'y_2'+32(1-x^2)^{-2}K^{-\frac{1}{3}}[\phi_1^{\frac{2}{3}}\phi_2^{\frac{1}{3}}-\phi_1^{-\frac{1}{3}}\phi_2^{\frac{1}{3}}-\phi_1^{\frac{2}{3}}\phi_2^{-\frac{2}{3}}+\phi_1^{-\frac{4}{3}}\phi_2^{-\frac{2}{3}}]=0,\\
&\label{equn_GBergerEinstein04}y_3''-2x^{-1}(1+2x^2)(1-x^2)^{-1}y_3'+\frac{1}{2}y_1'y_3'+32(1-x^2)^{-2}K^{-\frac{1}{3}}[\phi_1^{-\frac{1}{3}}\phi_2^{\frac{1}{3}}-\phi_1^{-\frac{1}{3}}\phi_2^{-\frac{2}{3}}-\phi_1^{\frac{2}{3}}\phi_2^{\frac{4}{3}}+\phi_1^{\frac{2}{3}}\phi_2^{-\frac{2}{3}}]=0,
\end{align}
for $y_i(x)\in C^{\infty}([0,1])$ for $i=1,2,3$ with the boundary condition
\begin{align}\label{equn_GBergerBV01}
\phi_1(0)=\frac{\lambda_2}{\lambda_1},\,\,\phi_2(0)=\frac{\lambda_3}{\lambda_2},\,\,K(1)=\phi_1(1)=\phi_2(1)=1,\,\,y_i'(0)=y_i'(1)=0,\,\,\,\text{for}\,\,i=1,2,3.
\end{align}
Combining $(\ref{equn_GBergerEinstein01})$ and $(\ref{equn_GBergerEinstein02})$ we have
\begin{align}\label{equn_GBergerEinstein05}
&(y_1')^2-[(y_2')^2+y_2'y_3'+(y_3')^2]-12x^{-1}(1+x^2)(1-x^2)^{-1}y_1'+48(1-x^2)^{-2}[3-2K^{-\frac{1}{3}}(\phi_1^2\phi_2)^{\frac{1}{3}}\\
&-2K^{-\frac{1}{3}}(\phi_1^{-1}\phi_2)^{\frac{1}{3}}-2K^{-\frac{1}{3}}(\phi_1\phi_2^2)^{-\frac{1}{3}}+K^{-\frac{1}{3}}\phi_1^{-\frac{4}{3}}\phi_2^{-\frac{2}{3}}+K^{-\frac{1}{3}}\phi_1^{\frac{2}{3}}\phi_2^{-\frac{2}{3}}+K^{-\frac{1}{3}}\phi_1^{\frac{2}{3}}\phi_2^{\frac{4}{3}}]=0\notag,
\end{align}
Recall that any three equations in the system of five equations $(\ref{equn_GBergerEinstein01})-(\ref{equn_GBergerEinstein04})$ and $(\ref{equn_GBergerEinstein05})$ containing at least one of $(\ref{equn_GBergerEinstein03})$ and $(\ref{equn_GBergerEinstein04})$, combining with the initial data imply the other two equations.

 The idea of symmetry extension approach in \cite{Li} is inspired by \cite{Wang}, see also \cite{Andersson-Dahl} and \cite{Anderson1}.

For $n=3$, assume $(M^4, g)$ is a CCE manifold with its conformal infinity $(\partial M, [\hat{g}])$. If we replace the non-positive sectional curvature condition by the condition that there exists a constant $\frac{1}{2} < \lambda \leq 1$ such that
\begin{align*}
Y(\partial M, [\hat{g}])\geq \lambda^{\frac{2}{3}}Y(\mathbb{S}^3,[g_0]),
\end{align*}
with $g_0$ the round sphere metric, then the upper bound estimate
 \begin{align}\label{ineqn_boundWeylHE1}
\sup_M |W|_g \leq T
\end{align}
still holds with some constant $T$ depending on $\lambda$, see Corollary 1.7 in \cite{LiQingShi}. The proof there is based on the control of the relative volume growth of
geodesic balls by the Yamabe constant at the conformal infinity and a blowing up argument relating to the relative volume growth. Moreover, we have the following curvature pinch estimates:

\begin{thm}\label{EHBoundary} (Theorem 1.6, \cite{LiQingShi}) For any $\epsilon >0$, there exists $\delta > 0$, for any conformally compact Einstein manifold $(M^{n+1}, g)$ ($n\geq 3$),
one gets
\begin{equation}\label{close-to-hyper}
|K [g] + 1| \leq \epsilon,
\end{equation}
for all sectional curvature $K$ of $g$, provided that
$$
Y(\partial M, [\hat{g}]) \geq (1-\delta) Y(\mathbb{S}^n, [g_{\mathbb{S}}]).
$$
Particularly, any conformally compact Einstein manifold with its conformal infinity of Yamabe constant sufficiently close to that of the round sphere is
necessarily negatively curved and also $(M, g)$ is simply connected.
\end{thm}

Recall that for any smooth metric $h\in[\hat{g}]$ at the conformal infinity, there exists a unique geodesic defining function $x$ about $h$ in a neighborhood of $\partial M$, see \cite{Graham}. For a CCE metric of $C^2$, based on \cite{Graham}, in \cite{CDLS} the authors proved the following regularity result.

\begin{thm}\label{thm_expansion1}
Assume $\overline{M}$ is a smooth compact manifold of dimension $n+1$, $n\geq 3$, with $M$ its interior and $\partial M$ its boundary. If $g$ is a conformally compact Einstein metric of class $C^2$ on $M$ with conformal infinity $(\partial M, [\gamma])$, and $\hat{g}\in [\gamma]$ is a smooth metric on $\partial M$. Then there exists  a smooth coordinates cover of $\overline{M}$ and a smooth geodesic defining $x$ corresponding to $\hat{g}$. Under this smooth coordinates cover, the conformal compactification $\bar{g}=x^2g$ is smooth up to the boundary for $n$ odd and has the expansion
\begin{align}\label{equn_expansion1}
\bar{g}=\,dx^2+g_x=\, dx^2+\hat{g}+x^2g^{(2)}+\,(\text{even powers})\,+x^{n-1} g^{(n-1)}+ x^{n}g^{(n)}+...
\end{align}
with $g^{(k)}$ smooth symmetric $(0,2)$-tensors on $\partial M$ such that for $2k<n$, $g^{(2k)}$ can be calculated explicitly inductively using the Einstein equations and $g^{(n)}$ is a smooth trace-free nonlocal term; while for $n$ even,  $\bar{g}$ is of class $C^{n-1}$, and more precisely it is polyhomogeneous and has the expansion
\begin{align}\label{equn_expansion2}
\bar{g}=\,dx^2+g_x=\, dx^2+\hat{g}+x^2g^{(2)}+\,(\text{even powers})\,+x^n\log(x) \tilde{g}+ x^{n}g^{(n)}+...
\end{align}
with $\tilde{g}$ and $g^{(k)}$ smooth symmetric $(0,2)$-tensors on $\partial M$, such that for $2k<n$, $g^{(2k)}$ and $\tilde{g}$ can be calculated explicitly inductively using the Einstein equations, $\tilde{g}$ is trace-free and $g^{(n)}$ is a smooth nonlocal term with its trace locally determined.
\end{thm}

\section{Uniqueness of Non-positively Curved Conformally Compact Einstein metrics with Generalized Berger Sphere Conformal Infinity}\label{section3}

 Recall that for a given Berger metric $\hat{g}$ on $\mathbb{S}^3$, we have proved the uniqueness of the non-positively curved CCE metrics $g$ with $(\mathbb{S}^3, [\hat{g}])$ as its conformal infinity (\cite{Li}).  In this section, for a given generalized Berger metric $\hat{g}$ on $\mathbb{S}^3$, we study the uniqueness of the non-positively curved CCE metrics $g$ satisfying $(\ref{eqn_Einstein})$ with $(\mathbb{S}^3, [\hat{g}])$ as its conformal infinity. By \cite{Li}, it is equivalent to show the uniqueness of the solution to the boundary value problem $(\ref{equn_GBergerEinstein01})-(\ref{equn_GBergerBV01})$.

 We start with the monotonicity of $y_i(x)$ $(i=1,2,3)$ for a global solution $(y_1,y_2,y_3)$ on $x\in[0,1]$. 
If either $\phi_1(0)=1$, or $\phi_2(0)=1$, or $\phi_1(0)\phi_2(0)=1$, then the metric $\hat{g}$ is a Berger metric, and the uniqueness of the solution to the boundary value problem has been proved in \cite{Li}. In this section we consider the generalized Berger metric $\hat{g}$ with $\phi_1(0)\neq 1$, $\phi_2(0)\neq 1$ and $\phi_1(0)\phi_2(0)\neq 1$. By the volume comparison theorem,
\begin{align*}
K(0)=\lim_{x\to 0}\frac{\text{det}(\bar{h})}{\text{det}(\bar{h}^{\mathbb{H}^4}(x))}=\lim_{r\to +\infty}\frac{\text{det}(g_r)}{\text{det}(g_r^{\mathbb{H}^4}(r))}<1,
\end{align*}
where
\begin{align*}
g^{\mathbb{H}^4}=dr^2+g_r^{\mathbb{H}^4}(r)=x^{-2}(dx^2+\frac{(1-x^2)^2}{4}\bar{h}^{\mathbb{H}^4})
\end{align*}
is the hyperbolic metric. Moreover it is proved in \cite{LiQingShi} that
\begin{align*}
(\frac{Y(\mathbb{S}^3,[\hat{g}])}{Y(\mathbb{S}^3,[g^{\mathbb{S}^3}])})^{\frac{3}{2}}\leq K(0)=\lim_{r\to +\infty}\frac{\text{det}(g_r)}{\text{det}(g_r^{\mathbb{H}^4}(r))},
\end{align*}
where $Y(\mathbb{S}^3,[\hat{g}])$ is the Yamabe constant of $(\mathbb{S}^3,[\hat{g}])$ and $g^{\mathbb{S}^3}$ is the round sphere metric.

From now on, without loss of generality, we assume $\lambda_1 > \max\{\lambda_2, \lambda_3\}$, so that $\phi_1(0) < 1$ and $\phi_1(0)\phi_2(0)<1$.

\begin{lem}\label{lem_monotonicity301}
For the initial data $\phi_1(0), \phi_2(0)> 0$, we have $y_1'(x)>0$ for $x\in(0,1)$. Moreover, if $\phi_1(0) < 1$, $\phi_1(0)\phi_2(0)<1$ and $1<\phi_1(0)+\phi_1(0)\phi_2(0)$,  
then $y_2'$, $y_3'$ and $y_2'+y_3'$ have no zeroes on $x\in(0,1)$. That is to say, $K$, $\phi_1$, $\phi_2$ and $\phi_1\phi_2$ are monotonic on $x\in(0,1)$.
\end{lem}
\begin{proof}
The proof is a modification of Lemma 5.1 in \cite{Li}. Notice that the zeroes of $y_i'$ ($1\leq i \leq 3$) are discrete on $x\in[0,1]$. Assume that there exists a zero of $y_1'$ on $x\in(0,1)$. Let $x_1$ be the largest zero of $y_1'$ on $x\in(0,1)$. Multiplying $x^{-1}(1-x^2)^2$ on both sides of $(\ref{equn_GBergerEinstein01})$ and integrating the equation on $x\in[x_1,1]$, we have
\begin{align}
&\label{equn_singularint21}(x^{-1}(1-x^2)^2y_1')'+x^{-1}(1-x^2)^2[\frac{1}{6}(y_1')^2+\frac{1}{3}((y_2')^2+y_2'y_3'+(y_3')^2)]=0,\\
&\int_{x_1}^1x^{-1}(1-x^2)^2[\frac{1}{6}(y_1')^2+\frac{1}{3}((y_2')^2+y_2'y_3'+(y_3')^2)]dx=0.\notag
\end{align}
Therefore, $y_1'=0$ on $x\in[x_1,1]$. Since $y_1$ is analytic, $y_1'=0$ for $x\in[0,1]$, contradicting with the fact $y_1(0)<y_1(1)$. Therefore, there is no zero of $y_1'$ on $x\in(0,1)$. Therefore, $y_1'>0$ for $x\in (0,1)$.

Denote $\phi_3=\phi_1\phi_2$ and $y_4=\log(\phi_3)$. Summarizing $(\ref{equn_GBergerEinstein03})$ and $(\ref{equn_GBergerEinstein04})$, we have
\begin{align}
y_4''-2x^{-1}(1+2x^2)(1-x^2)^{-1}y_4'+\frac{1}{2}y_1'y_4'+32(1-x^2)^{-2}K^{-\frac{1}{3}}[\phi_1^{\frac{2}{3}}\phi_2^{\frac{1}{3}}-\phi_1^{-\frac{1}{3}}\phi_2^{-\frac{2}{3}}-\phi_1^{\frac{2}{3}}\phi_2^{\frac{4}{3}}+\phi_1^{-\frac{4}{3}}\phi_2^{-\frac{2}{3}}]=0.
\end{align}
Let $x_i\in (0,1)$ be a zero of $y_i'$ for $i=2,3, 4$. We multiplying $x^{-2}(1-x^2)^3$ on both sides of $(\ref{equn_GBergerEinstein03})$, and integrate the equation on $x\in[x_2,1]$ to obtain
\begin{align}
&\label{equn_y2singularint}(x^{-2}(1-x^2)^3y_2')'+\frac{1}{2}x^{-2}(1-x^2)^3y_1'y_2'+ 32x^{-2}(1-x^2)K^{-\frac{1}{3}}[\phi_1^{\frac{2}{3}}\phi_2^{\frac{1}{3}}-\phi_1^{-\frac{1}{3}}\phi_2^{\frac{1}{3}}-\phi_1^{\frac{2}{3}}\phi_2^{-\frac{2}{3}}+\phi_1^{-\frac{4}{3}}\phi_2^{-\frac{2}{3}}]=0,\\
&\int_{x_2}^1\frac{1}{2}x^{-2}(1-x^2)^3y_1'y_2'dx= -\int_{x_2}^132x^{-2}(1-x^2)K^{-\frac{1}{3}}[\phi_1^{\frac{2}{3}}\phi_2^{\frac{1}{3}}-\phi_1^{-\frac{1}{3}}\phi_2^{\frac{1}{3}}-\phi_1^{\frac{2}{3}}\phi_2^{-\frac{2}{3}}+\phi_1^{-\frac{4}{3}}\phi_2^{-\frac{2}{3}}]dx.\notag
\end{align}
That is
\begin{align}\label{eqn_intright02}
&\int_{x_2}^1\frac{1}{2}x^{-2}(1-x^2)^3y_1'y_2'dx= -\int_{x_2}^132x^{-2}(1-x^2)K^{-\frac{1}{3}}\phi_1^{-\frac{4}{3}}\phi_2^{-\frac{2}{3}}(1-\phi_1)(1+\phi_1-\phi_3)dx.
\end{align}
Similarly,
\begin{align}
&\label{eqn_intright03}\int_{x_3}^1\frac{1}{2}x^{-2}(1-x^2)^3y_1'y_3'dx= -\int_{x_3}^132x^{-2}(1-x^2)K^{-\frac{1}{3}}\phi_1^{-\frac{1}{3}}\phi_2^{-\frac{2}{3}}(1-\phi_2)(-1+\phi_1+\phi_3)dx,\\
&\label{eqn_intright04}\int_{x_4}^1\frac{1}{2}x^{-2}(1-x^2)^3y_1'y_4'dx= -\int_{x_4}^132x^{-2}(1-x^2)K^{-\frac{1}{3}}\phi_1^{-\frac{4}{3}}\phi_2^{-\frac{2}{3}}(1-\phi_3)(1-\phi_1+\phi_3)dx.
\end{align}

We assume that $y_2'$ achieves the largest zero $x_2\in (0,1)$ in $\{y_2',y_3',y_4'\}$. Notice that $y_i'(1)=0$ and $y_i(1)=0$ for $i=2,3,4$. We have that $y_i'(1-\phi_{i-1})>0$ on $x\in(x_2,1)$ for $i=2,3,4$. By $(\ref{eqn_intright02})$, there exists a point $x\in(x_2,1)$ such that
\begin{align}\label{ineqnright13}
1+\phi_1(x)-\phi_3(x)<0.
\end{align}
We {\bf claim} that there is no zero of $y_3'$ and $y_4'$ on $x\in (0,1)$. If that is not the case, assume $y_3'$  achieves the largest zero $x_3 \in (0,x_2]$ in $\{y_3',y_4'\}$. Since $y_4'$ keeps the sign on $(x_3, 1)$, we have $\phi_3(x)>0$ and $y_4'<0$ on $(x_3, 1)$, contradicting with $(\ref{eqn_intright03})$. Otherwise, assume $y_4'$  achieves the largest zero $x_4 \in (0,x_2]$ in $\{y_3',y_4'\}$. Then $y_3'$ keeps the sign on $(x_2, 1)$ and by $(\ref{ineqnright13})$ we have $\phi_3=\phi_1\phi_2>\phi_1$ on $(x_2, 1)$ and therefore $\phi_2>1$ on $(x_4, 1)$, which implies $(1-\phi_1+\phi_3)>0$ on $(x_4,1)$, contradicting with $(\ref{eqn_intright04})$. That proves the {\bf claim}. By $(\ref{ineqnright13})$, $\phi_2>1$ and $\phi_3=\phi_1\phi_2>1$ on $x\in (0,1)$. Therefore, for $x\in(x_2,1)$,
\begin{align}
1+\phi_1(x)-\phi_3(x)> 1+ \phi_2^{-1} - \phi_3 > 1+ \phi_2^{-1}(0) - \phi_3(0)>0
\end{align}
where the last inequality is by the condition of the lemma, contradicting with $(\ref{eqn_intright02})$. Therefore, $y_2'$ could not achieve the largest zero on $x\in (0, 1)$ in $\{y_2',y_3',y_4'\}$.

Similar argument yields that neither $y_3'$ nor $y_4'$ could achieve the largest zero on $x\in (0, 1)$ in $\{y_2',y_3',y_4'\}$. Therefore, there is no zero of $y_i'$ on $x\in(0,1)$ for $i=2,3,4$. This completes the proof of the lemma.

\end{proof}
\begin{Remark}
Based on a integral comparison argument (see Lemma 5.3 and Theroem 5.4  in \cite{Li}), we can obtain the uniqueness of solutions of the boundary value problem of conformally compact Einstein metrics with the Berger sphere as its conformal infinity. But here, due to the interaction of the three quantities $y_1,\,y_2$ and $y_3$, such direct comparisons fail to conclude the uniqueness of solutions for the generalized Berger sphere case. To handle this difficulty, we have to do a priori estimates on the solutions and consider the total variation of the solutions, which refers to a global discussion.
\end{Remark}

Using the monotonicity of $y_i$ ($1\leq i \leq 4$), we now give a uniform estimate of $y_i$ on $x\in[0,1]$ under the condition in Lemma \ref{lem_monotonicity301}.

By $(\ref{equn_GBergerEinstein05})$ and the initial value condition, we have
\begin{align}
y_1'&\label{equn_y1lowerorder1}=6x^{-1}(1-x^2)^{-1}[1+x^2-\sqrt{(1+x^2)^2+\frac{1}{36}x^2(1-x^2)^2((y_2')^2+y_2'y_3'+(y_3')^2)-\frac{4}{3}x^2(3-\Upsilon(x))}]\\
&=6x^{-1}(1-x^2)^{-1}[1+x^2-\sqrt{(1-x^2)^2+\frac{1}{36}x^2(1-x^2)^2((y_2')^2+y_2'y_3'+(y_3')^2)+\frac{4}{3}x^2 \Upsilon(x)}\,\,]\notag,
\end{align}
where
\begin{align*}
\Upsilon(x)=K^{-\frac{1}{3}}[2(\phi_1^2\phi_2)^{\frac{1}{3}}+2(\phi_1^{-1}\phi_2)^{\frac{1}{3}}+2(\phi_1\phi_2^2)^{-\frac{1}{3}}-\phi_1^{-\frac{4}{3}}\phi_2^{-\frac{2}{3}}-\phi_1^{\frac{2}{3}}\phi_2^{-\frac{2}{3}}-\phi_1^{\frac{2}{3}}\phi_2^{\frac{4}{3}}].
\end{align*}
Since $y_1'>0$ for $x\in(0,1)$, by $(\ref{equn_y1lowerorder1})$ it is clear that
\begin{align}\label{inequn_boundaryYamabeconstant}
3-\Upsilon(x)> 0\,\, \text{for}\,\, x\in(0,1).
\end{align}
Recall that $\Upsilon\in C^{\infty}([0,1))$, hence we have $3 - \Upsilon(0)\geq 0$. Notice that by $(\ref{equn_GBergerEinstein02})-(\ref{equn_GBergerEinstein04})$,
\begin{align*}
&y_1''(0)= 4 (3- \Upsilon(0)),\\
&y_2''(0)=32K(0)^{-\frac{1}{3}}[\phi_1(0)^{\frac{2}{3}}\phi_2(0)^{\frac{1}{3}}-\phi_1(0)^{-\frac{1}{3}}\phi_2(0)^{\frac{1}{3}}-\phi_1(0)^{\frac{2}{3}}\phi_2(0)^{-\frac{2}{3}}+\phi_1(0)^{-\frac{4}{3}}\phi_2(0)^{-\frac{2}{3}}],\\
&y_3''(0)=32K(0)^{-\frac{1}{3}}[\phi_1(0)^{-\frac{1}{3}}\phi_2(0)^{\frac{1}{3}}-\phi_1(0)^{-\frac{1}{3}}\phi_2(0)^{-\frac{2}{3}}-\phi_1(0)^{\frac{2}{3}}\phi_2(0)^{\frac{4}{3}}+\phi_1(0)^{\frac{2}{3}}\phi_2(0)^{-\frac{2}{3}}].
\end{align*}
Therefore,
\begin{align*}
&\frac{d^2}{dx^2}(3-\Upsilon)\big|_{x=0}\\
&=\frac{1}{3}y_1''(0)\Upsilon(0)-\frac{2}{3}K(0)^{-\frac{1}{3}}\{[2(\phi_1^2\phi_2)^{\frac{1}{3}}-(\phi_1^{-1}\phi_2)^{\frac{1}{3}}-(\phi_1\phi_2^2)^{-\frac{1}{3}}+2\phi_1^{-\frac{4}{3}}\phi_2^{-\frac{2}{3}}-\phi_1^{\frac{2}{3}}\phi_2^{-\frac{2}{3}}-\phi_1^{\frac{2}{3}}\phi_2^{\frac{4}{3}}]y_2''(0)\\
&+[(\phi_1^2\phi_2)^{\frac{1}{3}}+(\phi_1^{-1}\phi_2)^{\frac{1}{3}}-2(\phi_1\phi_2^2)^{-\frac{1}{3}}+\phi_1^{-\frac{4}{3}}\phi_2^{-\frac{2}{3}}+\phi_1^{\frac{2}{3}}\phi_2^{-\frac{2}{3}}-2\phi_1^{\frac{2}{3}}\phi_2^{\frac{4}{3}}]y_3''(0)\}.
\end{align*}
We {\bf claim} that $(3-\Upsilon(0))>0$ and hence $y_1''(0)>0$. Otherwise, $(3-\Upsilon(0))=0$ and hence $y_1''(0)=0$. After substituting the expression of $y_i''(0)$, we have
\begin{align*}
&\frac{d^2}{dx^2}(3-\Upsilon)\big|_{x=0}\\
&=-128K^{-\frac{2}{3}}\phi_1^{-\frac{8}{3}}\phi_2^{-\frac{4}{3}}[(1-\phi_1)^2(1+\phi_1+\phi_1^2-\phi_1\phi_2-\phi_1^2\phi_2)+\phi_1^3\phi_2^2(1-\phi_1\phi_2)(1-\phi_2)]\big|_{x=0}\\
&=-128K^{-\frac{2}{3}}\phi_1^{-\frac{8}{3}}\phi_2^{-\frac{4}{3}}[(1-\phi_1\phi_2)^2((1-\phi_1)(1+\phi_1\phi_2)+\phi_1^2\phi_2^2)+\phi_1^3(\phi_2-1)(1-\phi_1)]\big|_{x=0}.
\end{align*}
Using the first identity when $\phi_2(0)<1$, while using the second identity when $\phi_2(0)>1$, combining with the fact $\phi_1(0)<1$ and $\phi_1(0)\phi_2(0)<1$, we have
\begin{align*}
&\frac{d^2}{dx^2}(3-\Upsilon)\big|_{x=0}<0.
\end{align*}
Since $\frac{d}{dx}(3-\Upsilon)\big|_{x=0}=0$, we have $\frac{d}{dx}(3-\Upsilon(x))<0$ for $x>0$ small, and therefore $(3-\Upsilon(x))<0$ for $x>0$ small, contradicting with $(\ref{inequn_boundaryYamabeconstant})$. The {\bf claim} is proved. This gives a lower bound of $K(0)$ under the condition of Lemma \ref{lem_monotonicity301}:
\begin{align}\label{ineqn_Kinitialbound301}
K(0)> \big(\frac{2(\phi_1^2\phi_2)^{\frac{1}{3}}+2(\phi_1^{-1}\phi_2)^{\frac{1}{3}}+2(\phi_1\phi_2^2)^{-\frac{1}{3}}-\phi_1^{-\frac{4}{3}}\phi_2^{-\frac{2}{3}}-\phi_1^{\frac{2}{3}}\phi_2^{-\frac{2}{3}}-\phi_1^{\frac{2}{3}}\phi_2^{\frac{4}{3}}}{3}\big)^3\big|_{x=0},
\end{align}
 and hence for $|1-\phi_i(0)|$ small, there exists a constant $C>0$ independent of the initial data and the solution, such that
 \begin{align}\label{ineqn_upperbddeterm301}
 |1-K(0)|\leq C(|1-\phi_1(0)|+|1-\phi_2(0)|).
 \end{align}
 But the condition in Lemma \ref{lem_monotonicity301}, we have
 \begin{align}\label{ineqn_lowerbk301}
 K(0)>\frac{1}{9}\phi_1(0)^{-1}\phi_2^{-2}(0)(1+\phi_1(0)\phi_2(0))^3,
 \end{align}
 when $\phi_2(0)\geq 1$; while for $\phi_2(0)<1$,
  \begin{align}\label{ineqn_lowerbk302}
 K(0)>\frac{1}{9}\phi_1(0)^{-1}\phi_2(0)(1+\phi_1(0))^3.
 \end{align}
 Now we give an estimate of the solution away from $x=1$.

\begin{lem}\label{lem_uniformests301}
Let $\phi_1(0) < 1$, $\phi_1(0)\phi_2(0)<1$ and $1<\phi_1(0)+\phi_1(0)\phi_2(0)$. If moreover we have $\phi_1(0),\,\phi_2(0)>\delta_0$ for some constant $\delta_0\in(0,1)$, then there exists a constant $C=C(\delta_0)>0$ independent of the solution and the initial data $\phi_i(0)$ such that
\begin{align}
&|y_1^{(k)}(x)|\leq C (|1-\phi_1(0)|+|1-\phi_2(0)|)x^{2-k},\\
&|y_i^{(k)}(x)| \leq C|1-\phi_{i-1}(0)|x^{2-k},
\end{align}
with $y_i^{(k)}$ the $k-$th order derivative of $x$, for $k=1,2$, $i=2,3$ and $x\in(0,\frac{3}{4}]$. The control still holds on the interval $(0, 1-\epsilon]$ for any $\epsilon>0$ small with some constant $C=C(\delta_0, \epsilon)>0$. 
\end{lem}
\begin{proof} By Lemma \ref{lem_monotonicity301} and the equation $(\ref{equn_y1lowerorder1})$, we have that for $x\in(0,1)$,
\begin{align}\label{ineqn_upperlr3011}
y_1'<6x^{-1}(1-x^2)^{-1}[1+x^2-\sqrt{(1-x^2)^2}\,\,]=12x(1-x^2)^{-1}.
\end{align}

Notice that
\begin{align*}
\min\{\phi_i(0), 1\} < \phi_i(x) < \max\{\phi_i(0), 1\},
\end{align*}
for $x\in (0,1)$ and $ 1\leq i \leq 3$, and $K(0)< K(x) <1$ for $x\in (0,1)$. By the interior estimates of the second order elliptic equations $(\ref{equn_GBergerEinstein01})-(\ref{equn_GBergerEinstein04})$ and the inequality $(\ref{ineqn_upperbddeterm301})$, there exists some constant $C=C(\delta_0)>0$ independent of the initial data and the solution so that
\begin{align*}
&|y_1^{(k)}(x)| \leq C(|1-\phi_1(0)|+|1-\phi_2(0)|),\,\,\,\text{and}\\
&|y_i^{(k)}(x)| \leq C(|1-\phi_{i-1}(0)|),
\end{align*}
for $2\leq i \leq 4$, $0 \leq k\leq 4$ and $x \in [\frac{1}{4}, \frac{3}{4}]$. To get global estimates, we multiply $x^{-5}(1-x^2)^6K^{\frac{1}{2}}$ to $(\ref{equn_GBergerEinstein02})$ and do integration on the interval $[x, \frac{3}{4}]$ to obtain
\begin{align*}
&(x^{-5}(1-x^2)^6K^{\frac{1}{2}}y_1')'+16x^{-5}(1-x^2)^4K^{\frac{1}{2}}[3-\Upsilon]=0,\\
&x^{-5}(1-x^2)^6K(x)^{\frac{1}{2}}y_1'(x)=(\frac{4}{3})^5(\frac{7}{16})^6K^{\frac{1}{2}}(\frac{3}{4})y_1'(\frac{3}{4})+\int_x^{\frac{3}{4}}16s^{-5}(1-s^2)^4K^{\frac{1}{2}}[3-\Upsilon(s)]ds,
\end{align*}
 for $x \in (0, \frac{3}{4}]$, and hence
\begin{align}
\label{ineqn_yd1left301}0 < y_1'(x) \leq C(|1-\phi_1(0)|+|1-\phi_2(0)|)x,
\end{align}
for $x\in(0,\frac{3}{4}]$, with some constant $C=C(\delta_0)>0$ independent of the initial data and the solution.

Similarly, by multiplying $x^{-2}(1-x^2)^3K^{\frac{1}{2}}$ to $(\ref{equn_GBergerEinstein03})$ and $(\ref{equn_GBergerEinstein04})$ and doing integration on $[x,\frac{3}{4}]$ correspondingly, we have
\begin{align}
&\label{ineqn_yd1left302}|y_2'(x)| \leq C|1-\phi_1(0)|x,\\
&\label{ineqn_yd1left303}|y_3'(x)| \leq C|1-\phi_2(0)|x,
\end{align}
for $x\in(0,\frac{3}{4})$, with some constant $C=C(\delta_0)>0$ independent of the solution and the initial data. Substituting these inequalities to $(\ref{equn_GBergerEinstein02})-(\ref{equn_GBergerEinstein04})$, we then have
\begin{align}
&|y_1''(x)|\leq C (|1-\phi_1(0)|+|1-\phi_2(0)|),\\
&|y_i''(x)| \leq C|1-\phi_{i-1}(0)|,
\end{align}
with $C=C(\delta_0)>0$ some constant independent of the initial data and the solutions, for $x\in (0,\frac{3}{4})$ and $i=2,3$.
\end{proof}

For $x\in(\frac{1}{2}, 1)$, multiplying $(x-x^3)$ on both sides of $(\ref{equn_GBergerEinstein03})$, we have
\begin{align}\label{eqn_regulary2302}
((x-x^3)y_2')'-(3+x^2)y_2'+\frac{1}{2}(x-x^3)y_1'y_2'+32x(1-x^2)^{-1}K^{-\frac{1}{3}}\phi_1^{-\frac{4}{3}}\phi_2^{-\frac{2}{3}}(1-\phi_1)(1+\phi_1-\phi_3)=0.
\end{align}
Since $y_1'$ and $y_2'$ keep the sign, we integrate the equation on $(\frac{1}{2}, x)$ and combine it with $(\ref{ineqn_upperlr3011})$ to have
\begin{align*}
&(x-x^3)|y_2'(x)|\leq C|y_2'(\frac{1}{2})|+C|y_2(x)-y_2(\frac{1}{2})| -C |1-\phi_1(0)|\log(1-x^2),\\
&|y_2'(x)|\leq C |1-\phi_1(0)|(1-x^2)^{-1}(1-\log(1-x^2)),
\end{align*}
with some constant $C=C(\delta_0)>0$ independent of the initial data and the solution, for $\frac{1}{2} \leq x < 1$. Similarly, we have
\begin{align*}
|y_3'(x)|\leq C |1-\phi_2(0)|(1-x^2)^{-1}(1-\log(1-x^2)),
\end{align*}
with some constant $C=C(\delta_0)>0$ independent of the initial data and the solution, for $\frac{1}{2} \leq x < 1$.

Now we use the bound of the Weyl tensor to give better estimates on $y_i^{(k)}$ with $ k=1, 2$ and $1\leq i \leq 3$ near $x=1$. Under the polar coordinates $(y^0,y^1,y^2,y^3)=(r,\theta^1,\theta^2,\theta^3 )$,
\begin{align*}
g=d(y^0)^2+\displaystyle\sum_{1\leq i, j\leq 3}g_{ij}dy^idy^j,
\end{align*}
and along the geodesic $\theta=\theta_0$,
\begin{align*}
g=d(y^0)^2+\displaystyle\sum_{1 \leq i \leq 3}g_{ii}d(y^i)^2.
\end{align*}
It is easy to obtain the calculations:
\begin{align*}
&\frac{d}{dr}g_{ii}=-\frac{x}{4}\frac{d}{dx}((x^{-1}-x)^2I_i)=\frac{1}{4x^2}(-x(1-x^2)^2\frac{d}{dx}I_i+2(1-x^4)I_i),\\
&g_{ii}^{-1}\frac{d}{dr}g_{ii}=(-xI_i^{-1}\frac{d}{dx}I_i+2(1+x^2)(1-x^2)^{-1}),
\end{align*}
for $i=1,2,3$. 
Recall that by the non-positivity of the sectional curvature of $g$ and the Einstein equation $(\ref{eqn_Einstein})$, we have that $|W|_g\leq T= 2 \sqrt{6}$. Using the boundedness of the Weyl tensor, we give an estimate of the solution on $x \in [\frac{3}{4}, 1]$.
\begin{lem}\label{lem_yiboundrightGB}
Assume that $|W|_g\leq \varepsilon$ with some constant $0<\varepsilon\leq T$. Under the condition of Lemma \ref{lem_monotonicity301}, and moreover we assume there exists a constant $\delta_0\in(0,1)$ such that $\phi_1(0),\phi_2(0)\geq \delta_0$. Then we have
\begin{align}
&\label{ineqn_yd1right301}|y_1^{(k)}|\leq C\,\varepsilon^2(1-x^2)^{4-k},\\
&\label{ineqn_yd1right302}|y_i^{(k)}|\leq C\,\varepsilon(1-x^2)^{2-k},
\end{align}
for $k=1,2$ and $x \in [\frac{3}{4}, 1]$, with $C=C(\delta_0)$.
\end{lem}
\begin{proof}
By the condition in the lemma, the following components of the Weyl tensor has the bound
\begin{align}\label{ineqn_Weylboundmixed}
|(g_{ii})^{-\frac{1}{2}}(g_{qq})^{-\frac{1}{2}}(g_{pp})^{-\frac{1}{2}}W_{piq0}(g)|\leq \varepsilon,
\end{align}
which is expressed as
\begin{align*}
&|(g_{ii})^{-\frac{1}{2}}(g_{qq})^{-\frac{1}{2}}(g_{pp})^{-\frac{1}{2}}W_{piq0}(g)|\\
&=|(g_{ii})^{-\frac{1}{2}}(g_{qq})^{-\frac{1}{2}}(g_{pp})^{\frac{1}{2}}[\frac{1}{2}\frac{d g_{ii}}{dr}(-(g_{ii})^{-1}+(g_{pp})^{-1}+(g_{pp})^{-1}g_{qq}(g_{ii})^{-1})\\ 
&+\frac{1}{2}\frac{dg_{pp}}{dr}((g_{pp})^{-1}-(g_{pp})^{-2}g_{ii}+(g_{pp})^{-2}g_{qq})-(g_{pp})^{-1}\frac{d g_{qq}}{dr}]|\\
&=\frac{x^2}{(1-x^2)}I_p^{\frac{1}{2}}I_q^{-\frac{1}{2}}I_i^{-\frac{1}{2}}|[I_i^{-1}\frac{d I_i}{dx}(-1+I_iI_p^{-1}+I_p^{-1}I_q)+I_p^{-1}\frac{d I_p}{dx}(1-I_iI_p^{-1}+I_p^{-1}I_q)-2(I_q^{-1}\frac{dI_q}{dx})I_p^{-1}I_q]|\\
&=2\frac{x^2}{(1-x^2)}I_q^{-\frac{1}{2}}|\frac{d}{dx}[I_i^{\frac{1}{2}}I_p^{-\frac{1}{2}}+I_i^{-\frac{1}{2}}I_p^{\frac{1}{2}}-I_i^{-\frac{1}{2}}I_p^{-\frac{1}{2}}I_q]|,
\end{align*}
for any $\{i, p, q\}=\{1,2,3\}$. Here we have used the symmetry of the metric $g$ for the derivatives $\frac{\partial}{\partial \theta^p}g_{ij}$, which can be found in \cite{Li}.

We choose $(i,p,q)=(1,2,3) $ and $(2,3,1)$. Now by $(\ref{ineqn_Weylboundmixed})$, we have
\begin{align*}
&(\phi_1-1-\phi_1\phi_2)y_2'-2\phi_1\phi_2y_3'=O(1)\,\varepsilon(1-x^2),\\
&2y_2'+(\phi_1\phi_2-\phi_1+1)y_3'=O(1)\,\varepsilon(1-x^2),
\end{align*}
for $x\in [\frac{3}{4}, 1]$, with $|O(1)|\leq \frac{32}{9}\phi_1^{\frac{1}{2}}(0)$. Therefore, using the initial values and the monotonicity of $\phi_i$, we have
\begin{align*}
&y_2'(x)=O(1)\,\varepsilon(1-x^2),\\
&y_3'(x)=O(1)\,\varepsilon(1-x^2),
\end{align*}
for $\frac{3}{4}\leq x \leq 1$ with $O(1)$ uniformly bounded, depending only on the lower bound $\delta_0$ of the initial data. 
We then integrate these two inequalities on the interval $(x,1)$, and use the monotonicity of $y_i$ and the fact $y_i(1)=0$ to have
\begin{align*}
&y_2(x)=O(1)\,\varepsilon(1-x^2)^2,\\
&y_3(x)=O(1)\,\varepsilon(1-x^2)^2,
\end{align*}
for $\frac{3}{4}\leq x \leq 1$ with $O(1)$ uniformly bounded, depending on the lower bound $\delta_0$. For $y_1'$, we multiply $x^{-1}(1-x^2)^2K^{\frac{1}{6}}$ on both sides of $(\ref{equn_GBergerEinstein01})$ and do integration on $(x, 1)$ for $\frac{3}{4}\leq x \leq 1$,
\begin{align*}
&(x^{-1}(1-x^2)^2K^{\frac{1}{6}}y_1')'+\frac{1}{3}x^{-1}(1-x^2)^2K^{\frac{1}{6}}[((y_2')^2+y_2'y_3'+(y_3')^2)]=0,\,\,\text{and}\,\,\\
&y_1'(x)=\frac{1}{3}K(x)^{-\frac{1}{6}}x(1-x^2)^{-2}\int_x^1s^{-1}(1-s^2)^2K^{\frac{1}{6}}(s)[((y_2')^2+y_2'y_3'+(y_3')^2)]ds\\
&\leq C \,\varepsilon^2 (1-x^2)^3,
\end{align*}
with $C=C(\delta_0)>0$, where for the last inequality we have used monotonicity of $y_i$ and the lower bound of $K$ given by $(\ref{ineqn_lowerbk301})$ and $(\ref{ineqn_lowerbk302})$. Integrate the inequality on $(x,1)$, we have
\begin{align*}
|y_1(x)|\leq C \,\varepsilon^2 (1-x^2)^4,
\end{align*}
with some constant $C=C(\delta_0)>0$, for $\frac{3}{4} \leq x \leq 1$.

\vskip0.2cm

 Substituting these estimates back to $(\ref{equn_GBergerEinstein01})-(\ref{equn_GBergerEinstein04})$, we have
\begin{align*}
&|y_1''(x)|\leq C \,\varepsilon^2(1-x^2)^2,\\
&|y_i''(x)| \leq C \,\varepsilon,
\end{align*}
for $i=1,2$ and $\frac{3}{4}\leq x\leq 1$, with some constant $C=C(\delta_0)>0$.

\end{proof}

Now we turn to the uniqueness discussion.

\begin{proof}[Proof of Theorem \ref{thm_Einsteineqnsbvpuniqueness1}]
By \cite{Li}, the uniqueness of the metric is equivalent to the uniqueness of the solution to the boundary value problem $(\ref{equn_GBergerEinstein01})-(\ref{equn_GBergerBV01})$.

Assume that we have two solutions $(y_{11}, y_{12}, y_{13})$ and $(y_{21}, y_{22}, y_{23})$ to the boundary value problem $(\ref{equn_GBergerEinstein01})-(\ref{equn_GBergerBV01})$, with $y_{ij}=\log(\phi_{i(j-1)})$ for $i=1,2$ and $j=2,3$ and $y_{i1}=\log(K_i)$ for $i=1,2$. Denote $z_i=y_{1i}-y_{2i}$ correspondingly for $1\leq i \leq 3$.

If these two solutions have the same conformal infinity $[\hat{g}]$ and the same non-local term $g^{(3)}$ in the expansion, then they coincide by \cite{Biquard2}. For $1 \leq i \leq 3$, by the Einstein equations, the zeroes of $z_i$ are discrete unless $z_i$ is identically zero.

 Let $\varepsilon=T=2\sqrt{6}$ and $\delta_0\in(0,1)$ be a given constant. By the non-positivity of the sectional curvature of $g$ and the Einstein equation $(\ref{eqn_Einstein})$, it holds that $|W|_g\leq T$ pointwisely on $M$. Assume that $\phi_1(0),\,\phi_2(0)\geq \delta_0$ and by assumption, $\phi_1(0),\,\phi_2(0)$ satisfy the condition in Lemma \ref{lem_monotonicity301}. We first consider the cases $i = 2, 3$. For $i=2,3$, on the domain
\begin{align*}
D_i^-=\{x\in(0,1] \big| z_i(x)\leq 0\},
 \end{align*}
 let $b_{i1}^-<...<b_{im_i}^-$ be the set of local minimum points of $z_i$ on $D_i^-$, and we pick up all the (maximal) non-increasing intervals of $z_i$ on $\overline{D}_i^-$ (the closure of $D_i^-$): $[a_{i1}^-,b_{i1}^-]\bigcup [a_{i2}^-,b_{i2}^-] \bigcup...\bigcup [a_{im_i}^-,b_{im_i}^-]$ such that $a_{i1}^-<b_{i1}^-<a_{i2}^-<...<a_{im_i}^-<b_{im_i}^-$ with $m_i$ some integer; while on the domain
 \begin{align*}
 D_i^+=\{x\in(0,1] \big| z_i(x)\geq 0\},
  \end{align*}
  let $b_{i1}^+<...<b_{in_i}^+$ be the set of local maximum points of $z_i$ on $D_i^+$ with $n_i$ some integer, and we pick up all the (maximal) non-decreasing intervals of $z_i$ on $\overline{D}_i^+$ (the closure of $D_i^+$):
  \begin{align*}
  [a_{i1}^+,b_{i1}^+]\bigcup [a_{i2}^+,b_{i2}^+] \bigcup...\bigcup [a_{in_i}^+,b_{in_i}^+]
  \end{align*}
  such that $a_{i1}^+<b_{i1}^+<a_{i2}^+<...<a_{in_i}^+<b_{in_i}^+$.  Since $z_i\in C^{\infty}([0,1])$ and $z_i'$ has finitely many zeroes, it is clear that $z_i$ is of bounded variation on $x\in[0,1]$. For an interval $[a,b] \subseteq [0,1]$, we denote $V_a^b(z_i)$ the total variation of $z_i$ on $x \in [a,b]$, and we denote $V(z_i)$ the total variation of $z_i$ on $x\in[0,1]$.

 Recall that $z_i(0)=z_i(1)=0$ for $i=2, 3$. By the mean value theorem, there exist zeroes of $z_i'$ on $x\in(0,1)$. And also for $1 \leq j \leq m_i$, either $z_i(a_{ij}^-)=0$, or $z_i(a_{ij}^-)\leq 0$ with $a_{ij}^-$ a local maximum of $z_i$ on $[0,1)$ and $z_i'(a_{ij}^-)=0$. Similarly, for $1 \leq j \leq n_i$, either $z_i(a_{ij}^+)=0$, or $z_i(a_{ij}^+)\geq 0$ with $a_{ij}^+$ a local minimum of $z_i$ on $[0,1)$ and $z_i'(a_{ij}^+)=0$. It is clear that for $i=2, 3$,
\begin{align*}
\frac{1}{2}V(z_i)&=\displaystyle\sum_{j=1}^{m_i}V_{a_{ij}^-}^{b_{ij}^-}(z_i)+\displaystyle\sum_{j=1}^{n_i}V_{a_{ij}^+}^{b_{ij}^+}(z_i)\\
&=\displaystyle\sum_{j=1}^{m_i}|z_i(b_{ij}^-) - z_i( a_{ij}^-)|\,+\,\displaystyle\sum_{j=1}^{n_i}|z_i(b_{ij}^+) - z_i(a_{ij}^+)|.
\end{align*}

 We substitute the two solutions to $(\ref{eqn_regulary2302})$, and take difference to have
\begin{align*}
0=&((x-x^3)z_2')'-(3+x^2)z_2'+\frac{1}{2}(x-x^3)(y_{11}'z_2'+z_1'y_{22}')\\
&+32x(1-x^2)^{-1}K_1^{-\frac{1}{3}}\phi_{11}^{-\frac{4}{3}}\phi_{12}^{-\frac{2}{3}}(1+\phi_{11}-\phi_{11}\phi_{12})(\phi_{21}-\phi_{11})\\
&+32x(1-x^2)^{-1}[K_1^{-\frac{1}{3}}\phi_{11}^{-\frac{4}{3}}\phi_{12}^{-\frac{2}{3}}(1+\phi_{11}-\phi_{11}\phi_{12})-K_2^{-\frac{1}{3}}\phi_{21}^{-\frac{4}{3}}\phi_{22}^{-\frac{2}{3}}(1+\phi_{21}-\phi_{21}\phi_{22})]\,(1-\phi_{21}),
\end{align*}
where the term
\begin{align*}
&|32x(1-x^2)^{-1}[K_1^{-\frac{1}{3}}\phi_{11}^{-\frac{4}{3}}\phi_{12}^{-\frac{2}{3}}(1+\phi_{11}-\phi_{11}\phi_{12})-K_2^{-\frac{1}{3}}\phi_{21}^{-\frac{4}{3}}\phi_{22}^{-\frac{2}{3}}(1+\phi_{21}-\phi_{21}\phi_{22})](1-\phi_{21})|\\
&\leq \frac{32}{3}x(1-x^2)^{-1}|1-\phi_{21}|\,\times\, K(0)^{-\frac{4}{3}}\phi_1^{-\frac{7}{3}}(0)\phi_2^{-\frac{5}{3}}(0)\,\times\,\big[\phi_1(0)\phi_2(0)|K_1-K_2|\\
&+K(0)\phi_2(0)\,(4+\phi_1(0)|1-\phi_2(0)|\,)\,|\phi_{11}-\phi_{21}|+K(0)\phi_1(0)(4+\phi_2\phi_1)|\phi_{12}-\phi_{22}|\big]\\
&\leq C x(1-x^2)^{-1}|1-\phi_{21}|\times (|z_1|+|z_2|+|z_3|),
\end{align*}
with some constant $C>0$ depending on the lower bound $\delta_0$ of $\phi_1(0)$ and here for $K(0)$ we mean the lower bound of $K_i(0)$ ($i=1,2$) obtained in $(\ref{ineqn_Kinitialbound301})$. Pick up $1\leq j \leq m_2$ (resp. $n_2$). Now we do integration of the equation on $[a_{2j}^{\pm}, b_{2j}^{\pm}]$ to have
\begin{align*}
&(x-x^3)z_2'( a_{2j}^{\pm})+3(z_2(b_{2j}^{\pm})-z_2(a_{2j}^{\pm}))+\int_{a_{2j}^{\pm}}^{b_{2j}^{\pm}}x^2z_2'(x)dx\\
=&\int_{a_{2j}^{\pm}}^{b_{2j}^{\pm}}\big[ \frac{1}{2}(x-x^3)(y_{11}'z_2'+z_1'y_{22}')+O(1)x(1-x^2)^{-1}|1-\phi_{21}|\times (|z_1|+|z_2|+|z_3|)\\
&+32x(1-x^2)^{-1}K_1^{-\frac{1}{3}}\phi_{11}^{-\frac{4}{3}}\phi_{12}^{-\frac{2}{3}}(1+\phi_{11}-\phi_{11}\phi_{12})(\phi_{21}-\phi_{11})\,\big]\,dx,
\end{align*}
with $O(1)$ uniformly bounded, depending only on $\delta_0$. It is clear that the three terms on the left hand side of the equation have the same sign.  On the right hand side the third term can not be controlled by the left hand side in general on the interval $x\in[0,1]$, while $(\phi_{21}-\phi_{11})$ has a different sign from the left hand side on the interval $[a_{2j}^{\pm}, b_{2j}^{\pm}]$. Therefore,
\begin{align}\label{ineqn_mainineqn301}
&3|z_2(b_{2j}^{\pm})-z_2(a_{2j}^{\pm})|\\
\leq &\int_{a_{2j}^{\pm}}^{b_{2j}^{\pm}}\big[ \frac{1}{2}(x-x^3)(|y_{11}'z_2'|+|z_1'y_{22}'|)+C_{21}x(1-x^2)^{-1}|1-\phi_{21}|\times (|z_1|+|z_2|+|z_3|)\big]\,dx,\notag
\end{align}
for some constant $C_{21}>0$ uniformly bounded  depending only on $\delta_0$. By $(\ref{ineqn_yd1left301})$ we have
\begin{align}\label{ineqn_y11d1left301}
0 < \frac{1}{2} (x-x^3) y_{11}'(x) \leq \frac{1}{2} C_2(|1-\phi_1(0)|+|1-\phi_2(0)|)x^2(1-x^2),
\end{align}
for $x\in(0,\frac{3}{4}]$; while by $(\ref{ineqn_yd1right301})$, we have
\begin{align}\label{ineqn_y11d1right301}
0\leq \frac{1}{2}(x-x^3) y_{11}'\leq \frac{1}{2} C_3\,\varepsilon^2 x (1-x^2)^4,
\end{align}
for $x \in [\frac{3}{4}, 1]$, with the constants $C_2, C_3>0$ depending on $\delta_0$. Also, by $(\ref{ineqn_yd1left302})$ and the monotonicity of $\phi_{i1}$ we have
\begin{align}
&\label{ineqn_y22boundleft302}|y_{i2}'(x)| \leq C_{24}|1-\phi_1(0)|x,\\
&x(1-x^2)^{-1}|1-\phi_{i1}(x)|\leq \frac{12}{7}|1-\phi_1(0)|,
\end{align}
for $x\in(0,\frac{3}{4})$ and $i=1,2$, with some constant $C_{24}=C_{24}(\delta_0)>0$; while by $(\ref{ineqn_yd1right302})$,
\begin{align}
&\label{ineqn_y22boundright302-1}|y_{i2}'(x)|\leq C_{25} \,\varepsilon(1-x^2),\\
&\label{ineqn_y22boundright302-123}x(1-x^2)^{-1}|1-\phi_{i1}(x)|\leq C_{25}\,\varepsilon (1-x^2),
\end{align}
for $x \in [\frac{3}{4}, 1]$ and $i=1,2$, with $C_{25}=C_{25}(\delta_0)>0$. Now we let $\varepsilon_0\in (0,1)$ be a constant satisfying that for $\varepsilon_0\leq x \leq 1$,
\begin{align}
&\label{ineqn_requy11d1left301e1}\frac{1}{2}\varepsilon^2  \big(1-x^2\big)^4C_3\,\leq \frac{1}{2},\,\,\,\,\frac{1}{2} \big(1-x^2\big)^2C_{25}\,\varepsilon\leq \frac{1}{4},\\
&\label{ineqn_requy11d1left302e2}C_{21}C_{25}(1-x^2)\varepsilon\leq \frac{1}{4}.
\end{align}
For that we let $\varepsilon_0\in (0,1)$ satisfy that
\begin{align}
&\label{ineqn_requy11d1left301e1-2}
(1-\varepsilon_0^2)\leq \min\{\frac{1}{(24C_3)^{\frac{1}{4}}},\,\frac{1}{(4\sqrt{6}C_{25})^{\frac{1}{2}}},\,\frac{1}{8\sqrt{6}\,C_{21}C_{25}}\}.
\end{align}

Similarly, now we consider $z_3$. Multiplying $(x-x^3)$ on both sides of $(\ref{equn_GBergerEinstein04})$, we have
\begin{align}\label{eqn_regulary2303}
((x-x^3)y_3')'-(3+x^2)y_3'+\frac{1}{2}(x-x^3)y_1'y_3'+32x(1-x^2)^{-1}K^{-\frac{1}{3}}\phi_1^{-\frac{1}{3}}\phi_2^{-\frac{2}{3}}(1-\phi_2)(-1+\phi_1+\phi_1\phi_2)=0.
\end{align}
We substitute the two solutions to $(\ref{eqn_regulary2303})$, and take difference of the two equations obtained to have
\begin{align*}
0=&((x-x^3)z_3')'-(3+x^2)z_3'+\frac{1}{2}(x-x^3)(y_{11}'z_3'+z_1'y_{23}')\\
&+32x(1-x^2)^{-1}K_1^{-\frac{1}{3}}\phi_{11}^{-\frac{1}{3}}\phi_{12}^{-\frac{2}{3}}(-1+\phi_{11}+\phi_{11}\phi_{12})(\phi_{22}-\phi_{12})\\
&+32x(1-x^2)^{-1}[K_1^{-\frac{1}{3}}\phi_{11}^{-\frac{1}{3}}\phi_{12}^{-\frac{2}{3}}(-1+\phi_{11}+\phi_{11}\phi_{12})-K_2^{-\frac{1}{3}}\phi_{21}^{-\frac{1}{3}}\phi_{22}^{-\frac{2}{3}}(-1+\phi_{21}+\phi_{21}\phi_{22})]\,(1-\phi_{22}),
\end{align*}
where the term
\begin{align*}
&|32x(1-x^2)^{-1}[K_1^{-\frac{1}{3}}\phi_{11}^{-\frac{1}{3}}\phi_{12}^{-\frac{2}{3}}(-1+\phi_{11}+\phi_{11}\phi_{12})-K_2^{-\frac{1}{3}}\phi_{21}^{-\frac{1}{3}}\phi_{22}^{-\frac{2}{3}}(-1+\phi_{21}+\phi_{21}\phi_{22})](1-\phi_{22})|\\
&\leq \frac{32}{3}x(1-x^2)^{-1}|1-\phi_{22}|\,\times\, K(0)^{-\frac{1}{3}}\phi_1^{-\frac{4}{3}}(0)\phi_2^{-\frac{5}{3}}(0)\,\times\,\big[\phi_1(0)\phi_2(0)|K_1-K_2|\\
&+K(0)\phi_2(0)\,(1+2\phi_1(0)(1+\phi_2(0))\,)\,|\phi_{11}-\phi_{21}|+K(0)\phi_1(0)(2+\phi_2\phi_1)|\phi_{12}-\phi_{22}|\big]\\
&\leq C x(1-x^2)^{-1}|1-\phi_{22}|\times (|z_1|+|z_2|+|z_3|),
\end{align*}
with some constant $C=C(\delta_0)>0$ and here for $K(0)$ we mean the lower bound of $K_i(0)$ ($i=1,2$) obtained in $(\ref{ineqn_Kinitialbound301})$. Pick up $1\leq j \leq m_3$ (resp. $n_3$). Now we do integration of the equation on $[a_{3j}^{\pm}, b_{3j}^{\pm}]$ to have
\begin{align*}
&(x-x^3)z_3'( a_{3j}^{\pm})+3(z_3(b_{3j}^{\pm})-z_3(a_{3j}^{\pm}))+\int_{a_{3j}^{\pm}}^{b_{3j}^{\pm}}x^2z_3'(x)dx\\
=&\int_{a_{3j}^{\pm}}^{b_{3j}^{\pm}}\big[ \frac{1}{2}(x-x^3)(y_{11}'z_3'+z_1'y_{23}')+O(1)x(1-x^2)^{-1}|1-\phi_{22}|\times (|z_1|+|z_2|+|z_3|)\\
&+32x(1-x^2)^{-1}K_1^{-\frac{1}{3}}\phi_{11}^{-\frac{1}{3}}\phi_{12}^{-\frac{2}{3}}(-1+\phi_{11}+\phi_{11}\phi_{12})(\phi_{22}-\phi_{12})\,\big]\,dx,
\end{align*}
with $O(1)$ uniformly bounded, depending only on $\delta_0$.  It is clear that the three terms on the left hand side of the equation have the same sign. Similarly as the case of $z_2$, the third term on the right hand side has a different sign from the left hand side on the interval $[a_{3j}^{\pm}, b_{3j}^{\pm}]$. Therefore,
\begin{align}\label{ineqn_mainineqn302}
&3|z_3(b_{3j}^{\pm})-z_3(a_{3j}^{\pm})|\\
\leq &\int_{a_{3j}^{\pm}}^{b_{3j}^{\pm}}\big[ \frac{1}{2}(x-x^3)(|y_{11}'z_3'|+|z_1'y_{23}'|)+C_{31}x(1-x^2)^{-1}|1-\phi_{22}|\times (|z_1|+|z_2|+|z_3|)\big]\,dx,\notag
\end{align}
for some constant $C_{31}=C_{31}(\delta_0)>0$ uniformly bounded. Also, by $(\ref{ineqn_yd1left303})$ we have
\begin{align}
&\label{ineqn_y23boundleft303}|y_{i3}'(x)| \leq C_{34}|1-\phi_2(0)|x,\\
&x(1-x^2)^{-1}|1-\phi_{i2}(x)|\leq \frac{12}{7}|1-\phi_2(0)|,
\end{align}
for $x\in(0,\frac{3}{4})$ and $i=1,2$, with some constant $C_{34}=C_{34}(\delta_0)>0$; while by $(\ref{ineqn_yd1right302})$,
\begin{align}
&\label{ineqn_y23boundright303}|y_{23}'(x)|\leq C_{35}\,\varepsilon(1-x^2),\\
&\label{ineqn_y23boundright303abc}x(1-x^2)^{-1}|1-\phi_{22}(x)|\leq C_{35}\,\varepsilon (1-x^2),
\end{align}
for $x \in [\frac{3}{4}, 1]$ and $i=1,2$, with $C_{35}=C_{35}(\delta_0)>0$.  Now we let $\varepsilon_0\in(0,1)$ be a constant satisfying that for $\varepsilon_0\leq x \leq 1$,
\begin{align}
&\label{ineqn_conditiony13epsphi301e1}\frac{1}{2} \big(1-x^2 \big)^2C_{35}\,\varepsilon\leq \frac{1}{4},\\
&\label{ineqn_conditiony13epsphi302e2}C_{31} C_{35} \,\varepsilon (1-x^2) \leq \frac{1}{4}.
\end{align}
For that we let $\varepsilon_0\in (0,1)$ satisfy that
\begin{align}
&\label{ineqn_conditiony13epsphi301e3}(1-\varepsilon_0^2) \leq \min\{\frac{1}{\big(4\sqrt{6}\,C_{35} \big)^{\frac{1}{2}}},\,\,\frac{1}{8\sqrt{6}\,C_{31}C_{35}}, \}.
\end{align}

We then turn to the estimate of $z_1$.

Without loss of generality, we assume that $z_1(0)\geq 0$. If $z_1(0)>0$, by $(\ref{equn_GBergerEinstein02})$, we have $z_1''(0)>0$, and hence $z_1>0$ and $z_1'>0$ on the first interval of monotonicity $(0, b_{11}^+)$. Otherwise, if $z_1(0)=0$, in the expansion of $g_x$, denote the non-local term $g^{(3)}=\frac{1}{4}\bar{h}^{(3)}=\text{diag}(-\phi_1^{-1}(0))(a_1+\phi_2^{-1}(0)a_2), a_1, a_2)$. It has been shown that if the two solutions admit the same initial data $\phi_1(0),\,\phi_2(0)$ and the nonlocal term $g^{(3)}$, then the two solutions coincide on $x\in[0,1]$, see \cite{Biquard2}. In the expansion  $z_1(x)=\displaystyle\sum_{k=1}z_1^{(k)}(0)x^k$ at $x=0$, the coefficients $z_1^{(k)}(0)$ can be calculated explicitly inductively using the Einstein equations and expressed by the data $\phi_1(0)$, $\phi_2(0)$ and $g^{(3)}(0)$. In the expansion of $z_1=y_{11}-y_{21}$ at $x=0$ by the equation $(\ref{equn_GBergerEinstein02})$, one has that the nonlocal term is involved starting from the term $z_1^{(5)}(0)$ and $z^{(k)}(0)=0$ for $k\leq 4$. Assume that these two solutions has different non-local terms $g^{(3)}$, then there exists $k\geq 5$ so that $z_1^{(k)}(0)\neq0$. With out loss of generality, when $z_1(0)=0$, we assume $z_1'>0$ on the first interval of monotonicity $(0, b_{11}^+)$ of $z_1$.

Let $0=x_0<x_1<...<x_{k_1}<x_{k_1+1}=1$ be all the local maximum points and local minimum points of $z_1$ on $x\in[0,1]$, with $k_1$ some integers. Therefore, for any $0\leq j \leq k_1$, we have that $z_1'$ keeps the sign on $x\in(x_j, x_{j+1})$ with possibly finitely many zeroes on the interval.

 Multiply $x(1-x^2)$ on both sides of $(\ref{equn_GBergerEinstein01})$, we have
\begin{align}\label{eqn_y1regular1301}
(x(1-x^2)y_1')'-2y_1'+\frac{1}{6}x(1-x^2)(y_1')^2+\frac{1}{3}x(1-x^2)((y_2')^2+y_2'y_3'+(y_3')^2)=0.
\end{align}
Substituting the two solutions into $(\ref{eqn_y1regular1301})$ and take difference of the two equations obtained, we have
\begin{align}\label{eqn_z1regular1301}
0=&(x(1-x^2)z_1')'-2z_1'+\frac{1}{6}x(1-x^2)(y_{11}'+y_{21}')z_1'\\
&+\frac{1}{3}x(1-x^2)\,[(y_{12}')^2+y_{12}'y_{13}'+(y_{13}')^2-(y_{22}')^2-y_{22}'y_{23}'-(y_{23}')^2]\notag.
\end{align}
For each $0\leq j \leq k_1$, we do integration of $(\ref{eqn_z1regular1301})$ on the interval $x\in [x_j,x_{j+1}]$,
\begin{align}\label{eqn_z1regularint1301}
&2(z_1(x_{j+1})- z_1(x_j)) -\frac{1}{6}\,\int_{x_j}^{x_{j+1}}\,x(1-x^2)(y_{11}'+y_{21}')z_1' dx\\
=&\frac{1}{3}\int_{x_j}^{x_{j+1}}x(1-x^2)\,[(y_{12}')^2+y_{12}'y_{13}'+(y_{13}')^2-(y_{22}')^2-y_{22}'y_{23}'-(y_{23}')^2]dx,\notag\\
=&\frac{1}{3}\int_{x_j}^{x_{j+1}}x(1-x^2)\,[(y_{12}'+y_{22}'+y_{13}')z_2'+ (y_{13}'+y_{23}'+y_{22}')z_3']dx.\notag
\end{align}
By $(\ref{ineqn_y11d1left301})$ and $(\ref{ineqn_y11d1right301})$, we have
\begin{align}\label{ineqn_y11d1left301-1}
0 < x(1-x^2)(y_{11}'+y_{21}') \leq 2 C_2(|1-\phi_1(0)|+|1-\phi_2(0)|)x^2(1-x^2),
\end{align}
for $x\in(0,\frac{3}{4}]$, and
\begin{align}\label{ineqn_y11d1right301-1}
0\leq x(1-x^2)(y_{11}'+y_{21}')\leq 2 C_3\varepsilon^2x(1-x^2)^4,
\end{align}
for $x \in [\frac{3}{4}, 1]$, with the constants $C_2=C_2(\delta_0)>0$ and $ C_3=C_3(\delta_0)>0$. Also, by $(\ref{ineqn_y22boundleft302})$ and $(\ref{ineqn_y23boundleft303})$, we have
\begin{align}
&x(1-x^2)|y_{12}'+y_{22}'+y_{13}'| \leq (2 C_{24}|1-\phi_1(0)| + C_{34}|1-\phi_2(0)|)x^2(1-x^2),\\
&x(1-x^2)\,|y_{13}'+y_{23}'+y_{22}'| \leq ( C_{24}|1-\phi_1(0)| + 2 C_{34}|1-\phi_2(0)|)x^2(1-x^2),
\end{align}
for $x\in(0,\frac{3}{4})$, with some constants $C_{24}=C_{24}(\delta_0)$ and $C_{34}=C_{34}(\delta_0)$ independent of the solution and the initial data; while by $(\ref{ineqn_y22boundright302-1})$ and $(\ref{ineqn_y23boundright303})$, we have
\begin{align}
&x(1-x^2)|y_{12}'+y_{22}'+y_{13}'| \leq (2 C_{25} + C_{35})\,\varepsilon x(1-x^2)^2,\\
&x(1-x^2)\,|y_{13}'+y_{23}'+y_{22}'|\leq ( C_{25} + 2C_{35})\,\varepsilon x(1-x^2)^2,
\end{align}
for $x \in [\frac{3}{4}, 1]$, with some constants $C_{25}=C_{25}(\delta_0)>0$ and $C_{35}=C_{35}(\delta_0)>0$.

Now we let $\varepsilon_0\in(0,1)$ be a constant satisfying that for $\varepsilon_0\leq x \leq 1$,
\begin{align}
&\label{ineqn_conditiony11epsphi301e1-1}\frac{1}{3}\times \big(1-x^2\big)^4 C_3\,\varepsilon^2\leq 1,\\
&\label{ineqn_conditiony11epsphi302e2-2}(1-x^2)\leq \min\{\,\big(\varepsilon\,(2C_{25}+C_{35})\big)^{-\frac{1}{2}},\,\big(\varepsilon\,(C_{25}+2C_{35})\big)^{-\frac{1}{2}}\}.
\end{align}
For that we let $\varepsilon_0\in (0,1)$ satisfy that
\begin{align}
\label{ineqn_conditiony11epsphi302e2-2-2}(1-\varepsilon_0^2)\leq \min\big\{\,\big(8\,C_3\big)^{-\frac{1}{4}},\,\big(2\sqrt{6}\,(2C_{25}+C_{35})\big)^{-\frac{1}{2}},\,\big(2\sqrt{6}\,(C_{25}+2C_{35})\big)^{-\frac{1}{2}}\big\}.
\end{align}
We now choose $\varepsilon_0\in(0,1)$ to be the smallest number which satisfies $(\ref{ineqn_requy11d1left301e1-2})$, $(\ref{ineqn_conditiony13epsphi301e3})$ and $(\ref{ineqn_conditiony11epsphi302e2-2-2})$. 
By the monotonicity of $y_i$ on $x\in [0,1]$ and  the interior estimates of the second order elliptic equations $(\ref{equn_GBergerEinstein01})-(\ref{equn_GBergerEinstein04})$ on $x\in[\frac{3}{4}, \varepsilon_0]$, there exists some constant $C_4=C_4(\delta_0,\varepsilon_0)>0$ independent of the initial data and the solution so that
\begin{align*}
&|y_1^{(k)}(x)| \leq C_4(|1-\phi_1(0)|+|1-\phi_2(0)|),\,\,\,\text{and}\\
&|y_i^{(k)}(x)| \leq C_4(|1-\phi_{i-1}(0)|),
\end{align*}
for $2\leq i \leq 4$, $0 \leq k\leq 4$ and $ x \in [ \frac{3}{4}, \varepsilon_0]$.

 Now we assume that $\phi_1(0)$ and $\phi_2(0)$ satisfy that
\begin{align}
&\label{ineqn_requy11d1left301}\frac{1}{8}C_2(|1-\phi_1(0)|+|1-\phi_2(0)|)\leq \frac{1}{2},\,\,\,\,\,\,\,\,\,\,\frac{1}{8}C_{24}|1-\phi_1(0)|\leq \frac{1}{4},\\
&\label{ineqn_requy11d1left302aa2}\frac{\varepsilon_0}{1-\varepsilon_0^2}|1-\phi_1(0)| C_{21}\leq \frac{1}{4},\,\,\,\,\,\,\,\,\,\frac{1}{2} \times \frac{21}{64}C_4(|1-\phi_1(0)|+|1-\phi_2(0)|)\leq \frac{1}{2}.
\end{align}
Combining the conditions $(\ref{ineqn_requy11d1left301})-(\ref{ineqn_requy11d1left302aa2})$ on $\phi_1(0)$ and $\phi_2(0)$,  using the control $(\ref{ineqn_y11d1left301})- (\ref{ineqn_y22boundright302-123})$, and by the choice of $\varepsilon_0\in(0,1)$, we apply the inequality $(\ref{ineqn_mainineqn301})$ to have
\begin{align*}
&3|z_2(b_{2j}^{\pm})-z_2(a_{2j}^{\pm})|\leq \frac{1}{2}|z_2(b_{2j}^{\pm})-z_2(a_{2j}^{\pm})|+ \frac{1}{4}V_{a_{2j}^{\pm}}^{b_{2j}^{\pm}}(z_1)+ \frac{1}{4}(b_{2j}^{\pm}-a_{2j}^{\pm})\sup_{x\in(0,1)}(|z_1|+|z_2|+|z_3|),\\
&\frac{5}{2}|z_2(b_{2j}^{\pm})-z_2(a_{2j}^{\pm})|\leq \frac{1}{4}V_{a_{2j}^{\pm}}^{b_{2j}^{\pm}}(z_1)+ \frac{1}{4}(b_{2j}^{\pm}-a_{2j}^{\pm})(V(z_1)+V(z_2)+V(z_3)).
\end{align*}
Summarizing this inequality for all $j$, we have
\begin{align*}
\frac{5}{4}V(z_2)\leq \frac{1}{4}V(z_1)+ \frac{1}{4}(V(z_1)+V(z_2)+V(z_3)),
\end{align*}
and hence
\begin{align}\label{ineqn_variationz2bound302}
V(z_2)\leq \frac{1}{2}V(z_1)+ \frac{1}{4}V(z_3).
\end{align}
If in addition, we assume $\phi_2(0)$ satisfies that
\begin{align}\label{ineqn_conditiony13epsphi301}
\frac{1}{8}C_{34}|1-\phi_2(0)|\leq \frac{1}{4},\,\,\,\,\,\frac{\varepsilon_0}{1-\varepsilon_0^2}|1-\phi_2(0)|\leq \frac{1}{4},
\end{align}
then by the choice of $\varepsilon_0>0$ and the control $(\ref{ineqn_y23boundleft303})-(\ref{ineqn_y23boundright303abc})$, substituting the conditions $(\ref{ineqn_requy11d1left301})-(\ref{ineqn_requy11d1left302aa2})$ and $(\ref{ineqn_conditiony13epsphi301})$ on $\phi_1(0)$ and $\phi_2(0)$ to the inequality $(\ref{ineqn_mainineqn302})$, we have
\begin{align*}
&3|z_3(b_{3j}^{\pm})-z_3(a_{3j}^{\pm})|\leq \frac{1}{2}|z_3(b_{3j}^{\pm})-z_3(a_{3j}^{\pm})|+ \frac{1}{4}V_{a_{3j}^{\pm}}^{b_{3j}^{\pm}}(z_1)+ \frac{1}{4}(b_{3j}^{\pm}-a_{3j}^{\pm})\sup_{x\in(0,1)}(|z_1|+|z_2|+|z_3|),\\
&\frac{5}{2}|z_3(b_{3j}^{\pm})-z_3(a_{3j}^{\pm})|\leq \frac{1}{4}V_{a_{3j}^{\pm}}^{b_{3j}^{\pm}}(z_1)+ \frac{1}{4}(b_{3j}^{\pm}-a_{3j}^{\pm})(V(z_1)+V(z_2)+V(z_3)).
\end{align*}
Summarizing this inequality for all $j$, we have
\begin{align*}
\frac{5}{4}V(z_3)\leq \frac{1}{4}V(z_1)+ \frac{1}{4}(V(z_1)+V(z_2)+V(z_3)),
\end{align*}
and hence
\begin{align}\label{ineqn_variationz3bound303}
V(z_3)\leq \frac{1}{2}V(z_1)+ \frac{1}{4}V(z_2).
\end{align}

Now for the estimate of the total variation of $z_1$ on $x\in[0,1]$, we give a third condition on the initial data $\phi_1(0)$ and $\phi_2(0)$:
\begin{align}
&\label{ineqn_conditiony11epsphi301}\frac{1}{3}\times \frac{21}{64}C_4 (|1-\phi_1(0)|+|1-\phi_2(0)|) \leq 1,\,\,\,\,\,\frac{21}{64}C_4(2 |1-\phi_1(0)| + |1-\phi_2(0)|)\leq 1,\\
&\frac{21}{64}C_4( |1-\phi_1(0)| + 2|1-\phi_2(0)|)\leq 1,\,\,\,\,\,\frac{1}{4}(2 C_{24}|1-\phi_1(0)| + C_{34}|1-\phi_2(0)|)\leq 1,\\
&\label{ineqn_conditiony11epsphi302}\frac{1}{4}( C_{24}|1-\phi_1(0)| + 2 C_{34}|1-\phi_2(0)|)\leq 1.
\end{align}
 Then by $(\ref{eqn_z1regularint1301})$, we have
\begin{align}
|z_1(x_{j+1})- z_1(x_j)|\leq \frac{1}{3}(V_{x_j}^{x_{j+1}}(z_2)+V_{x_j}^{x_{j+1}}(z_3)).
\end{align}
Summarizing this inequality for all $j$, one has
\begin{align}\label{ineqn_variationz1bound301}
V(z_1)\leq \frac{1}{3}(V(z_2)+V(z_3)).
\end{align}

We summarize the above argument: Let $\varepsilon_0\in (0,1)$ be a constant defined below $(\ref{ineqn_conditiony11epsphi302e2-2-2})$. Assume that the initial data $\phi_1(0)$ and $\phi_2(0)$ satisfy the assumptions $(\ref{ineqn_requy11d1left301})-(\ref{ineqn_requy11d1left302aa2})$, $(\ref{ineqn_conditiony13epsphi301})$ and $(\ref{ineqn_conditiony11epsphi301})-(\ref{ineqn_conditiony11epsphi302})$, then by $(\ref{ineqn_variationz2bound302})$, $(\ref{ineqn_variationz3bound303})$ and $(\ref{ineqn_variationz1bound301})$, and the non-negativity of $V(z_i)$ ($1\leq i \leq 3$), one has
\begin{align*}
V(z_i)=0,
\end{align*}
for $1\leq i\leq 3$. Therefore, $z_i=0$ for $1\leq i\leq 3$. The uniqueness of the solution is proved. In particular, when we take $\delta_0=\frac{63}{64}$, we can choose $(1-\varepsilon_0^2)=0.000025$, $C_4=3\times 10^7$, and $\eta_0= 1-3\times 10^{-8}$ (this is not the optimal constant in the calculation and it could be smaller). If one could estimate $|W|_g$ to be small, then 
 by the same approach of the above estimate, $\eta_0\in(0,1)$ in the theorem could be much smaller. This completes the proof of Theorem \ref{thm_Einsteineqnsbvpuniqueness1}.

 \end{proof}

\vskip0.2cm

\section{Uniqueness and Existence of Non-positively Curved Conformally Compact Einstein Metrics with  High Dimensional Homogeneous Sphere Conformal Infinity}\label{section4}

A Riemannian manifold $(N, g)$ is called a {\it homogeneous Riemannian manifold}, if there exists a group $G$ of isometries acting on $(N, g)$ transitively, and hence $g$ is a $G$-invariant metric. Notice that we have the diffeomorphism $N \cong G/H$, with $H$ the isotropy group of some point on $N$. Any $G$-invariant metric $g$ on $N$ has the structure described as follows (see for instance \cite{Ziller1}\cite{Deng}). Let $\mathfrak{g}, \mathfrak{h}$ be the Lie algebra of $G$ and $H$. $\mathfrak{g}$ is equivalent to the Lie algebra of the Killing vector fields of the metric $g$ on $N$. Moreover, if $H$ is compact, $\mathfrak{g}$ has an $\text{ad}\mathfrak{h}$ invariant splitting $\mathfrak{g}=\mathfrak{h}\oplus \mathfrak{p}$ such that $[\mathfrak{h},\mathfrak{p}]\subseteq \mathfrak{p}$. $H$ acts on $\mathfrak{p}$ by the adjoint map which induces a splitting:
\begin{align}\label{eqn_splittingmetric}
\mathfrak{p}=\mathfrak{p}_0\oplus \mathfrak{p}_1\oplus...\oplus\mathfrak{p}_p,
\end{align}
where $\mathfrak{p}_i$ is an irreducible subspace for $1\leq i\leq p$, and $H$ acts on $\mathfrak{p}_0$ trivially. Let $B$ be a bi-invariant metric on $G$. Any $G$-invariant metric $g$ on $N\cong G/H$ is determined by its value on $\mathfrak{p}$, which has a splitting
\begin{align*}
g=h\big|_{\mathfrak{p}_0}+\displaystyle\sum_{i=1}^p\alpha_iB\big|_{\mathfrak{p}_i}
\end{align*}
with $h\big|_{\mathfrak{p}_0}$ an arbitrary metric on $\mathfrak{p}_0$, and any $\alpha_i>0$.

In this paper, for a homogeneous space we always mean a homogeneous Riemannian space.

The homogeneous metrics on spheres are $G$-invariant metrics under the transitive action of some Lie group $G$ on the spheres.  It is well known that Lie Groups acting effectively and transitively on spheres have been classified by D. Montgomery and H. Samelson (\cite{MS}), and A. Borel (\cite{Borel1}, \cite{Borel2}), see \cite{Besse} (p. 179) and also \cite{Ziller1} for instance.

Up to a scaling factor and isometry, homogeneous metrics on spheres are in one of the three classes: a one parameter family of $\text{U}(k+1)$-invariant metrics (equivalently, $\text{SU}(k+1)$-invariant) on $\mathbb{S}^{2k+1}\cong \text{SU}(k+1)/ \text{SU}(k)$ $(k\geq1)$, a three parameter family of $\text{Sp}(k+1)$-invariant metrics on $\mathbb{S}^{4k+3}\cong \text{Sp}(k+1)/\text{Sp(k)}$ (containing the $\text{SU}(2k+2)$-invariant metrics as a subset in these dimensions) with $k\geq0$, and a one parameter family of $\text{Spin}(9)$-invariant metrics on $\mathbb{S}^{15}\cong \text{Spin}(9)/ \text{Spin}(7)$. For more details, see \cite{Ziller1}. The Killing vector fields of the homogeneous metrics determine the symmetry structure of the $G$-invariant metrics, and generate the group action of $G$ on the spheres. 


Recall that $\mathfrak{p}=\mathfrak{p}_0\oplus \mathfrak{p}_1$ for the first two cases, where the adjoint action of $H$ on $\mathfrak{p}_0$ is trivial and $\mathfrak{p}_1$ is an invariant subspace which is irreducible under the adjoint action of $H$. Here $\text{dim}(\mathfrak{p}_0)=1$ for $\mathbb{S}^{2k+1} \cong \text{SU}(k+1)/\text{SU}(k)$ and the action of $\text{SU}(k)$ on $\mathfrak{p}_1$ is the usual action of $\text{SU}(k)$ on $\mathbb{C}^k$, while $\text{dim}(\mathfrak{p}_0)=3$ for $\mathbb{S}^{4k+3} \cong \text{Sp}(k+1)/\text{Sp}(k)$ and the action of $\text{Sp}(k)$ on $\mathfrak{p}_1$ is the usual action of $\text{Sp}(k)$ on $\mathbb{H}^k$. For the third case, $\mathfrak{p}=\mathfrak{p}_1\oplus \mathfrak{p}_2$ with $\text{dim}(\mathfrak{p}_1)=7$ and $\text{dim}(\mathfrak{p}_2)=8$, where the isotropy representations on $\mathfrak{p}_1$ and $\mathfrak{p}_2$ are the unique irreducible representations of $\text{Spin}(7)$ in that dimension.

Let $(M^{n+1}, g)$ be a CCE manifold which is Hadamard with its conformal infinity $(\mathbb{S}^n, [\hat{g}])$ so that $(\mathbb{S}^n, \hat{g})$ is a homogeneous space. Let $p_0 \in M$ be the center of gravity, $r$ be the distance function to $p_0$ on $(M, g)$ and $x=e^{-r}$ the geodesic defining function about $C\hat{g}$ with some constant $C>0$ as discussed in Section \ref{Sect:preliminary}. Now we pick up a point $q\in \partial M=\mathbb{S}^{n}$. Let $\gamma$ be the geodesic connecting $q$ and $p_0$. To choose the polar coordinate $(r, \theta)$ centered at $p_0$ near $\gamma$, we turn to the choice of the local coordinates $\theta=(\theta^1,...,\theta^n)$ on $\mathbb{S}^n$ near the point $q$ based on the above decomposition of the Lie algebra. 
For a 
given transitive $G$-action on $\mathbb{S}^n$, we can choose a basis of $G$-invariant action fields $Y_1,...,Y_n\in \mathfrak{p}$ 
and a coordinate $(\theta^1,...,\theta^n)$ near $q$ so that $d\theta^i=\sigma_i$ at $q$ with $\sigma_i$ the $1$-form corresponding to $Y_i$ such that under the coordinate $(\theta^1,..,\theta^n)$ each $G$-invariant metric $h$ has the diagonal form 
\begin{align*}
h=\text{diag}\big(I_1, I_2,...,I_2\big)
\end{align*}
at $q$, with $I_1,\,I_2$ some positive numbers, for $\mathbb{S}^{n}\cong \text{SU}(k+1)/\text{SU}(k)$ with $G=\text{SU}(k+1)$ and $n=2k+1$; while
\begin{align*}
h=\text{diag}\big(I_1, I_2, I_3, I_4, ...,I_4\big)
\end{align*}
with $I_1,...,I_4$ some positive numbers, up to an $\text{SO}(3)$ rotation in the subspace $\mathfrak{p}_0$ generated by $Y_1,\,Y_2$ and $Y_3$ for $\mathbb{S}^{n}\cong \text{Sp}(k+1)/\text{Sp}(k)$ with $G=\text{Sp}(k+1)$ and $n=4k+3$, see \cite{Ziller1}.  Indeed,  we consider the natural embedding $\mathbb{S}^n\subseteq \mathbb{R}^{n+1}$ with $\mathbb{S}^n$ the unit sphere on the Euclidean space $\mathbb{R}^{n+1}$, under the coordinates $(x_1,y_1;x_2,y_2;...;x_{\frac{n+1}{2}}, y_{\frac{n+1}{2}})$.  Without loss of generality, we assume that $q=(1,0,...,0)$, and choose $(\theta^1,...,\theta^n)=(y_1,x_2,y_2,...,x_{\frac{n+1}{2}},y_{\frac{n+1}{2}})$ near $q$. Now assume the $\text{SU}(k+1)$-invariant metric $\hat{g}$ on $\mathbb{S}^n$ with $n=2k+1$ has the standard diagonal form
\begin{align}\label{eqn_initialdataSUn}
\hat{g}=\text{diag}\big(\lambda_1, \lambda_2,...,\lambda_2\big)
\end{align}
at $q$ which is $\theta_0= (0,...,0)$ under the coordinate $\theta=(\theta^1,...,\theta^{2k+1})$, with $\lambda_1, \lambda_2 >0$; while the $\text{Sp}(k+1)$-invariant metric $\hat{g}$ on $\mathbb{S}^n$ with $n=4k+3$ has the standard diagonal form
\begin{align}\label{eqn_initialdataSpn}
\hat{g}=\text{diag}\big(\lambda_1, \lambda_2, \lambda_3, \lambda_4, ...,\lambda_4\big)
\end{align}
at $q$ which is $\theta_0= (0,...,0)$ under the coordinate $\theta=(\theta^1,...,\theta^{4k+3})$, with $\lambda_i >0$, $1\leq i\leq 4$. 
 Then by the extension of the Killing vector fields and the smoothness result in Theorem \ref{thm_expansion1}, under the polar coordinate $(r,\theta)$, the $\text{SU}(k+1)$-invariant metric $g_r$ on $\partial B_r(p_0)$ for $r\geq 0$ has the form
\begin{align}\label{eqn_metricSUn}
g_r=\sinh^{2}(r)\,\,\text{diag}\big(I_1(x), I_2(x),...,I_2(x)\big)
\end{align}
at the points along the geodesic $\gamma=\{(r,\theta_0)\big|r\geq 0\}= \{(r,0,...,0)\big|r\geq 0\}$, for some positive functions $I_1(x), I_2(x) \in C^{\infty}([0,1])$, and the Ricci curvature of $g_r$ has the expression
\begin{align}\label{eqn_RicciSUn}
\text{Ric}(g_r)=\text{diag}\big((n-1)I_1^2(x)I_2^{-2}(x), (n+1)-2I_1(x)I_2^{-1}(x),...,(n+1)-2I_1(x)I_2^{-1}(x)\big),
\end{align}
along $\gamma$, with respect to the conformal infinity $(\mathbb{S}^n, [\hat{g}])$ for $n=2k+1$;
while the $\text{Sp}(k+1)$-invariant metric $g_r$ on $\partial B_r(p_0)$ for $r\geq 0$ has the form
\begin{align}\label{eqn_metricSpn}
g_r=\sinh^{2}(r)\,\,\text{diag}\big(I_1(x), I_2(x), I_3(x), I_4(x), ..., I_4(x)\big),
\end{align}
along $\gamma=\{(r,\theta_0)\big|r\geq 0\}$, for some positive functions $I_i(x) \in C^{\infty}([0,1])$ ($1\leq i\leq 4$), and the Ricci curvature of $g_r$ has the expression (see \cite{Ziller1})
\begin{align}\label{eqn_RicciSpn}
\text{Ric}(g_r)=\text{diag}\big(\,\,&(4nt_1^2+\frac{2(t_1^2-(t_2-t_3)^2)}{t_2t_3}), (4nt_2^2+\frac{2(t_2^2-(t_1-t_3)^2)}{t_1t_3}), (4nt_3^2+\frac{2(t_3^2-(t_1-t_2)^2)}{t_1t_2}),\\
& 4n+8-2(t_1+t_2+t_3),\, ...,\, 4n+8-2(t_1+t_2+t_3)\big),
\end{align}
along $\gamma$, where $t_i=\frac{I_i(x)}{I_4(x)}$ ($1\leq i\leq 3$),  with respect to the conformal infinity
$(\mathbb{S}^n, [\hat{g}])$ for $n=4k+3$.

 Now we give a way to use the symmetry extension to derive $(\ref{eqn_metricSUn})$ and $(\ref{eqn_RicciSUn})$ from the initial data $(\ref{eqn_initialdataSUn})$ and the Einstein equations as in \cite{Li}. 
 Notice that the extended Killing vector fields $X_j$ in $(M, g)$ with respect to $Y_j$ satisfy $X_j=\displaystyle\sum_{m=1}^nX_j^m\frac{\partial}{\partial \theta^m}=\sum_{m=1}^nY_j^m\frac{\partial}{\partial \theta^m}$ on $(M,g)$, with the components $X_j^m$ independent of $r$ for $j=1,...,n$, under the polar coordinate $(r, \theta)=(r, \theta^1,...,\theta^n)$. And also
\begin{align*}
g =dr^2+g_r = dr^2 + \sum_{i,j=1}^ng_{ij}d\theta^id \theta^j.
\end{align*}
Therefore
\begin{align*}
&\frac{\partial}{\partial \theta^i}(X_q^mg_{mj})+\frac{\partial}{\partial \theta^j}(X_q^mg_{mi})-2\Gamma_{ij}^p(g)X_q^mg_{mp}=0,\,\,\,\text{which is}\\
&\frac{\partial}{\partial \theta^i}X_q^pg_{pj}+\frac{\partial}{\partial \theta^j }X_q^pg_{pi}+X_q^m\frac{\partial}{\partial \theta^m}g_{ij}=0.
\end{align*}
Define the inverse of the matrix with elements $X_i^j$ for $1\leq i, j \leq n$ as
\begin{align*}
\left(\begin{matrix}&&&\\ &Z_i^j&&\\ &&& \end{matrix}\right)=\left(\begin{matrix}&&&\\ &X_i^j&&\\ &&& \end{matrix}\right)^{-1}.
\end{align*}
We denote
\begin{align*}
C_{ij}^p=Z_i^q\frac{\partial}{\partial \theta^j}X_q^p,
\end{align*}
and
\begin{align}\label{equn_twotensor}
T_{ij}^p = - T_{ji}^p=C_{ij}^p- C_{ji}^p=Z_i^qZ_j^m[X_m, X_q]^p,
\end{align}
with $[X_m, X_q]$ the Lie bracket of $X_m$ and $X_q$.
Notice that $C_{ij}^p$ and $T_{ij}^p$ are independent of $r$ and the metric. Then
\begin{align}
&\label{equn_tangentspace}\frac{\partial}{\partial \theta^q}g_{ij}=-C_{qi}^mg_{mj}-C_{qj}^mg_{mi},\\
&\Gamma_{ij}^p(g_r)=\frac{1}{2}[-(C_{ij}^p+C_{ji}^p)+g^{pq}(-C_{iq}^mg_{mj}-C_{jq}^mg_{mi}+C_{qi}^mg_{mj}+C_{qj}^mg_{mi})].
\end{align}
The Ricci curvature of $g_r$ on the geodesic spheres has the expression (see \cite{Li}) 
\begin{align}
R_{ij}(g_r)&\label{equn_Riccitensor}=\frac{1}{2}(\frac{\partial}{\partial \theta^p}T_{ij}^p+C_{ip}^qT_{qj}^p+C_{qj}^pT_{ip}^q+C_{qp}^pT_{ji}^q)-\frac{1}{2}g^{pq}(\frac{\partial}{\partial \theta^p}T_{iq}^m+C_{ip}^sT_{sq}^m+C_{pq}^sT_{is}^m-C_{ps}^mT_{iq}^s)g_{mj}\\
&-\frac{1}{2}g^{pq}(\frac{\partial}{\partial \theta^p}T_{jq}^m+C_{jp}^sT_{sq}^m+C_{pq}^sT_{js}^m-C_{ps}^mT_{jq}^s)g_{mi}+\frac{1}{4}T_{pi}^sT_{sj}^p-\frac{1}{4}g^{pq}T_{pi}^sT_{sq}^mg_{mj}\notag\\
&-\frac{1}{4}g^{pq}T_{pj}^sT_{sq}^mg_{mi}-\frac{1}{2}g^{sq}T_{ps}^pT_{iq}^mg_{mj}-\frac{1}{2}g^{sq}T_{ps}^pT_{jq}^mg_{mi}+\frac{1}{4}g^{pq}T_{pj}^sT_{iq}^mg_{sm}+\frac{1}{4}g^{pq}T_{pi}^sT_{jq}^mg_{sm}\notag\\
&-\frac{1}{4}g^{pl}(T_{jl}^mg_{ms}+T_{sl}^mg_{mj})g^{sq}(T_{pq}^mg_{mi}+ T_{iq}^mg_{mp})\notag.
\end{align}

Let
\begin{align*}
g_r=\sinh^2(r)\bar{h}=\frac{x^{-2}(1-x^2)^2}{4}\bar{h}_{ij}d\theta^id\theta^j.
\end{align*}
 Therefore, at the points along the line $\gamma=\{(r, \theta_0)\big| r>0\}$, the Einstein equations are equivalent to the equations (see \cite{Li})
 \begin{align}
&\label{equn_EinsteinODEs1}\frac{d}{d x}(x(1-x^2)\bar{h}^{pq}\frac{\partial}{\partial x}\bar{h}_{pq})+\frac{1}{2}x(1-x^2)\bar{h}^{ps}\frac{d}{d x}\bar{h}_{sm}\bar{h}^{mq}\frac{d}{d x}\bar{h}_{qp}-2\bar{h}^{pq}\frac{d}{d x} \bar{h}_{pq}= 0,\\ 
&\label{equn_EinsteinODEs2}C_{pq}^p\bar{h}^{qs}\frac{d}{d x}\bar{h}_{si}+C_{iq}^m\bar{h}^{pq}\frac{d}{d x}\bar{h}_{mp}-C_{qp}^p\bar{h}^{qs}\frac{d}{d x}\bar{h}_{si}-C_{pi}^s\bar{h}^{pq}\frac{d}{d x}\bar{h}_{sq}=0,\\ \notag \\
&\label{equn_EinsteinODEs3}-\frac{1}{8}x(1-x^2)^2\frac{d^2}{d x^2}\bar{h}_{ij}+\frac{1}{8}[(n-1)+(1+n)x^2]\,(1-x^2)\frac{d}{d x}\bar{h}_{ij}+\frac{x(1-x^2)^2}{8}\bar{h}^{pq}\frac{d}{d x}\bar{h}_{pi}\frac{d}{d x}\bar{h}_{qj}\\
&+\frac{1}{8}(1+x^2)(1-x^2)\bar{h}^{pq}\frac{d}{d x}\bar{h}_{pq}\bar{h}_{ij}-\frac{1}{16}x(1-x^2)^2\bar{h}^{pq}\frac{d}{d x}\bar{h}_{pq}\frac{d}{d x}\bar{h}_{ij}+(1-n)x\bar{h}_{ij}+xR_{ij}(\bar{h})=0,\notag
\end{align}
for $1\leq i,j\leq n$, with $C_{ij}^p(\theta_0)$ independent of $x$, and $R_{ij}(\bar{h})=R_{ij}(g_r)$. By the regularity result in Theorem \ref{thm_expansion1} (see also \cite{CDLS}), $\bar{h}\in C^{\infty}([0,1])$. This is a system of ordinary differential equations for $\bar{h}_{ij}$ on $x\in [0,1]$, with the boundary conditions
\begin{align}\label{equn_boundaryvalue1}
\bar{h}_{ij}(0)\in [\hat{g}_{ij}(\theta_0)],\,\,\bar{h}_{ij}(1)=g^0_{ij},\,\,\frac{d}{d x}\bar{h}_{ij}(0)=0,\,\,\frac{d}{d x}\bar{h}_{ij}(1)=0,
\end{align}
with $g^0$ the round metric on $\mathbb{S}^n$.

 Now we calculate $C_{ij}^p$ and $\frac{\partial}{\partial \theta^p}T_{ij}^p$ at $\theta=\theta_0$ for $\mathbb{S}^5=\text{SU}(3)/\text{SU}(2)$. Consider the Euclidean space $\mathbb{R}^6=\mathbb{C}^3$ as the three dimensional complex space with the coordinate $(z_1,z_2,z_3)$, where $z_j=x_j+iy_j$ ($j=1,2,3$). At $q=(1,0,0)$, we choose a basis of the Lie algebra $su(3)$ (see for example \cite{KW})
\begin{align*}
v_1=\left(\begin{matrix}&2i&0&0\\ &0&-i&0\\ &0&0&-i \end{matrix}\right),\,\,v_2=\left(\begin{matrix}&0&1&0\\ &-1&0&0\\ &0&0&0 \end{matrix}\right),\,\,v_3=\left(\begin{matrix}&0&i&0\\ &i&0&0\\ &0&0&0 \end{matrix}\right),\\
v_4=\left(\begin{matrix}&0&0&1\\ &0&0&0\\ &-1&0&0 \end{matrix}\right),\,\,v_5=\left(\begin{matrix}&0&0&i\\ &0&0&0\\ &i&0&0 \end{matrix}\right),\,\,v_6=\left(\begin{matrix}&0&0&0\\ &0&i&0\\ &0&0&-i \end{matrix}\right),\\
v_7=\left(\begin{matrix}&0&0&0\\ &0&0&1\\ &0&-1&0 \end{matrix}\right),\,\,v_8=\left(\begin{matrix}&0&0&0\\ &0&0&i\\ &0&i&0 \end{matrix}\right).
\end{align*}
It is clear that the subspace generated by $\{v_6,v_7,v_8\}$ is tangent to $\left(\begin{matrix}&1&&\\ &&\text{SU}(2)& \end{matrix}\right)\subset \text{SU}(3)$ and is isomorphic to $su(2)$. Notice that $v_1\in \mathfrak{p}_0$ is $\text{Ad}_{\text{SU}(2)}$-invariant, and so is the subspace $\mathfrak{p}_2$ spanned by $\{v_2,v_3,v_4,v_5\}$. The corresponding $\text{SU(3)}$-invariant vector field $Y_j$ with respect to $v_j$ has the expression $Y_j=\left(\begin{matrix}z_1\,\,z_2\,\,z_3 \end{matrix}\right)v_j^T$ at the point $(z_1,z_2,z_3)\in \mathbb{S}^5$, $1\leq j\leq 8$. Therefore,
\begin{align*}
&Y_1=\left(\begin{matrix}&2i z_1\\ &-i z_2\\ &-i z_3 \end{matrix}\right)^T,\,\,Y_2=\left(\begin{matrix}&z_2\\ &-z_1\\ &0 \end{matrix}\right)^T,\,\,Y_3=\left(\begin{matrix}&i z_2\\ &i z_1\\ &0 \end{matrix}\right)^T,\,\,Y_4=\left(\begin{matrix}&z_3\\ &0\\ &-z_1 \end{matrix}\right)^T,\,\,Y_5=\left(\begin{matrix}&i z_3\\ &0\\ &i z_1 \end{matrix}\right)^T,\\
&Y_6=\left(\begin{matrix}&0\\ &i z_2\\ &-i z_3 \end{matrix}\right)^T,\,\,Y_7=\left(\begin{matrix}&0\\ &z_3\\ &-z_2 \end{matrix}\right)^T,\,\,Y_8=\left(\begin{matrix}&0\\ &i z_3\\ &i z_2 \end{matrix}\right)^T.
\end{align*}
at the point $(z_1,z_2,z_3)\in \mathbb{S}^5$. Direct calculations show that
\begin{align}
&\label{eqn_LieBraket501}[Y_1,Y_2]=-3Y_3,\,\,[Y_1,Y_3]=3Y_2,\,\,[Y_1,Y_4]=-3Y_5,\,\,[Y_1,Y_5]=3Y_4,\\
&[Y_2,Y_3]=-Y_1+Y_6,\,\,[Y_2,Y_4]=Y_7,\,\,[Y_2,Y_5]=Y_8,\,\,[Y_3,Y_4]=-Y_8,\notag\\
&[Y_3,Y_5]=Y_7,\,\,[Y_4,Y_5]=-Y_1-Y_6\notag.
\end{align}
We choose $Y_1,...,Y_5$ as a basis of the Killing vector field at $q$. Under the coordinate $(\theta^1,...,\theta^5)=(y_1,x_2,y_2,x_3,y_3)$ near $q$, we have the expression of $X_j=Y_j$ ($1\leq j\leq 8$),
\begin{align*}
&X_1= \left(\begin{matrix}2\sqrt{1-\sum_{k=1}^5(\theta^k)^2},\,\theta^3,\,-\theta^2,\,\theta^5,\,-\theta^4 \end{matrix}\right),\,\,X_2= \left(\begin{matrix}\theta^3,\,-\sqrt{1-\sum_{k=1}^5(\theta^k)^2},\,-\theta^1,\,0,\,0 \end{matrix}\right),\\
&X_3= \left(\begin{matrix}\theta^2,\,-\theta^1,\,\sqrt{1-\sum_{k=1}^5(\theta^k)^2},\,0,\,0 \end{matrix}\right), \,\,\,X_4= \left(\begin{matrix}\theta^5,\,0,\,0,\,-\sqrt{1-\sum_{k=1}^5(\theta^k)^2},\,-\theta^1 \end{matrix}\right),\\
&X_5= \left(\begin{matrix}\theta^4,\,0,\,0,\,-\theta^1,\,\sqrt{1-\sum_{k=1}^5(\theta^k)^2} \end{matrix}\right),\,\,\,X_6= \left(\begin{matrix}0,\,-\theta^3,\,\theta^2,\,\theta^5,\,-\theta^4 \end{matrix}\right),\\
&X_7= \left(\begin{matrix}0,\,\theta^4,\,\theta^5,\,-\theta^2,\,-\theta^3 \end{matrix}\right),\,\,\,X_8= \left(\begin{matrix}0,\,-\theta^5,\,\theta^4,\,-\theta^3,\,\theta^2 \end{matrix}\right).
\end{align*}
Hence at $\theta=\theta_0$,
\begin{align*}
&C_{12}^3=-C_{13}^2=C_{14}^5=-C_{15}^4=-\frac{1}{2},\\
&C_{21}^3=-C_{23}^1=-C_{31}^2=C_{32}^1=C_{41}^5=-C_{51}^4=C_{54}^1=-C_{45}^1=1,
\end{align*}
and $C_{ij}^p=0$ otherwise. Substituting to $(\ref{equn_EinsteinODEs2})$, we have
\begin{align*}
 \frac{d}{dx}\bar{h}_{ij}=0,
\end{align*}
for $i\neq j$ and $0 \leq x \leq 1$ and moreover
\begin{align*}
\frac{d}{dx}\bar{h}_{ii}=\frac{d}{dx}\bar{h}_{jj},
\end{align*}
for $i,j>1$. Using the initial data $(\ref{equn_boundaryvalue1})$ and $(\ref{eqn_initialdataSUn})$, we have that $\bar{h}_{ij}=0$ for $i\neq j$ and hence denote $\bar{h}$ as
\begin{align*}
\bar{h}=\text{diag}(I_1(x),\,I_2(x),\,...,I_2(x)),
\end{align*}
for $0 \leq x \leq 1$ at $\theta=\theta_0$, where $I_i(x)\in C^{\infty}([0,1])$, by Theorem \ref{thm_expansion1}. By $(\ref{equn_twotensor})$ and $(\ref{eqn_LieBraket501})$, we have
\begin{align*}
\frac{\partial}{\partial \theta^m}T_{ij}^p&=\,-\alpha_{cba}(-Z_i^d\frac{\partial}{\partial \theta^m}X_d^q\,Z_q^cZ_j^bX_a^p-Z_i^cZ_j^d\frac{\partial}{\partial \theta^m}X_d^q\,Z_q^bX_a^p+Z_i^cZ_j^b\frac{\partial}{\partial \theta^m}X_a^p)\\
&=\,-\alpha_{cba}(-C_{im}^q\,Z_q^cZ_j^bX_a^p-Z_i^cC_{jm}^q\,Z_q^bX_a^p+Z_i^cZ_j^bX_a^qC_{qm}^p),
\end{align*}
with $\alpha_{cba}$ the structure constants in $(\ref{eqn_LieBraket501})$. Therefore, for $0 \leq x \leq 1$ and $\theta=\theta_0$, the expression of $Ric(\bar{h})=Ric(g_r)$ in $(\ref{eqn_RicciSUn})$ holds for $k=2$. Based on the choice of the basis of the $\text{SU}(k+1)$-invariant vector fields on $\mathbb{S}^{2k+1}$ for $k=2$ in \cite{KW} and $k=3$ in \cite{SSHF}, by induction one can easily get the general formula of the choice of the basis of the $\text{SU}(k+1)$-invariant vector fields on $\mathbb{S}^{2k+1}$ for general $k\geq 2$, and do the above calculations to get the expression $(\ref{eqn_metricSUn})$ of $g_r$ and $(\ref{eqn_RicciSUn})$ for $Ric(g_r)$. For $\mathbb{S}^{4k+3}\cong \text{Sp}(k+1)/\text{Sp}(k)$, similar calculations can be done and the basis of Lie algebra $\text{sp}(k+1)$ is chosen in \cite{Ziller1}. To deal with the possible $\text{SO}(3)$ rotation in $\mathfrak{p}_0$, one has to view $(\ref{equn_EinsteinODEs2})$ as a system of $1$-order linear homogeneous ODEs of $\bar{h}_{ij}$ $(1\leq i < j \leq 3)$ and solve the initial value problem of $(\ref{equn_EinsteinODEs2})$ with homogeneous initial data as the generalized Berger metric case, for details see Lemma 4.2 in \cite{Li}. And hence $(\ref{eqn_metricSpn})$ holds and the calculations of $(\ref{eqn_RicciSpn})$ can be found in \cite{Ziller1}.

On $(M^{n+1}, g)$ for $n=2k+1$ with its conformal infinity $(\mathbb{S}^{2k+1}, [\hat{g}])$ where $\hat{g}$ is an $\text{SU}(k+1)$-invariant metric, let $(r, \theta)$ be the polar coordinate chosen above. Assume $\hat{g}$ has the form $(\ref{eqn_initialdataSUn})$ at $q$ under the local coordinate $\theta=(\theta^1,...,\theta^n)$. We substitute $(\ref{eqn_metricSUn})$ and $(\ref{eqn_RicciSUn})$ to the Einstein equations $(\ref{equn_EinsteinODEs1})-(\ref{equn_EinsteinODEs3})$ to have
\begin{align*}
&\frac{d}{d x}[x(1-x^2)\,(I_1^{-1}\frac{\partial}{\partial x}I_1+\,(n-1)I_2^{-1}\frac{\partial}{\partial x}I_2)]+\frac{1}{2}x(1-x^2)((I_1^{-1}\frac{\partial}{\partial x}I_1)^2+\,(n-1)(I_2^{-1}\frac{\partial}{\partial x}I_2)^2)\\
&-2(I_1^{-1}\frac{\partial}{\partial x}I_1+\,(n-1)I_2^{-1}\frac{d}{d x}I_2)= 0,\\
&-\frac{1}{8}x(1-x^2)^2\frac{d^2}{d x^2}I_1+\frac{1}{8}[(n-1)+(1+n)x^2]\,(1-x^2)\frac{d}{d x}I_1+\frac{x(1-x^2)^2}{8}I_1^{-1}(\frac{d}{d x}I_1)^2\\
&+\frac{1}{8}(1+x^2)(1-x^2)(I_1^{-1}\frac{d}{d x}I_1+\,(n-1)I_2^{-1}\frac{d}{d x}I_2)I_1\\
&-\frac{1}{16}x(1-x^2)^2(I_1^{-1}\frac{d}{d x}I_1+\,(n-1)I_2^{-1}\frac{d}{d x}I_2)\frac{d}{d x}I_1+(1-n)xI_1\,+\,(n-1)xI_1^2I_2^{-2}=0,\\
&-\frac{1}{8}x(1-x^2)^2\frac{d^2}{d x^2}I_2+\frac{1}{8}[(n-1)+(1+n)x^2]\,(1-x^2)\frac{d}{d x}I_2+\frac{x(1-x^2)^2}{8}I_2^{-1}(\frac{d}{d x}I_2)^2\\
&+\frac{1}{8}(1+x^2)(1-x^2)(I_1^{-1}\frac{d}{d x}I_1+\,(n-1)I_2^{-1}\frac{d}{d x}I_2)I_2\\
&-\frac{1}{16}x(1-x^2)^2(I_1^{-1}\frac{d}{d x}I_1+\,(n-1)I_2^{-1}\frac{d}{d x}I_2)\frac{d}{d x}I_2+(1-n)xI_2+x(n+1-2I_1I_2^{-1})=0,
\end{align*}
on $x\in [0,1]$ where $I_i'=\frac{d}{dx}I_i$, with the boundary condition
\begin{align}\label{equn_boundaryvaluediagn1}
\frac{I_1(0)}{I_2(0)}=\frac{\lambda_1}{\lambda_2},\,\,I_1(1)=I_2(1)=1,\,\,I_1'(0)=I_2'(0)=I_1'(1)=I_2'(1)=0.
\end{align}
Denote $K=I_1I_2^{n-1}$ and $\phi=\frac{I_2}{I_1}$, $y_1=\log(K)$ and $y_2=\log(\phi)$ so that
\begin{align}\label{equn_changevariablesSUn}
I_1=(K\phi^{1-n})^{\frac{1}{n}},\,\,I_2=(K\phi)^{\frac{1}{n}}.
\end{align}
Therefore, the boundary value problem of the Einstein metrics
becomes
\begin{align}
&\label{equn_SUnEinstein01}y_1''+\frac{1}{2n}[(y_1')^2+(n-1)(y_2')^2]-x^{-1}(1+3x^2)(1-x^2)^{-1}y_1'=0,\\
&\label{equn_SUnEinstein02}y_1''-[2n-1+(1+2n)x^2\,]\,x^{-1}(1-x^2)^{-1}y_1'+\frac{1}{2}(y_1')^2\\
&+8(n-1)(1-x^2)^{-2}[n-(n+1)K^{-\frac{1}{n}}\phi^{-\frac{1}{n}}+K^{-\frac{1}{n}}\phi^{-\frac{n+1}{n}}]=0,\notag\\
&\label{equn_SUnEinstein03}y_2''-[(n-1)+(1+n)x^2\,]\,x^{-1}(1-x^2)^{-1}y_2'+\frac{1}{2}y_1'y_2'+8(n+1)(1-x^2)^{-2}K^{-\frac{1}{n}}\phi^{-\frac{1}{n}}(\phi^{-1}-1)=0,
\end{align}
for $y_1(x),y_2(x)\in C^{\infty}([0,1])$ with the boundary condition
\begin{align}\label{equn_SUnBV01}
\phi(0)=\frac{\lambda_2}{\lambda_1},\,\,K(1)=\phi(1)=1,\,\,y_1'(0)=y_2'(0)=y_1'(1)=y_2'(1)=0.
\end{align}
Combining $(\ref{equn_SUnEinstein01})$ and $(\ref{equn_SUnEinstein02})$, we have
\begin{align}\label{equn_SUnEinstein04}
(y_1')^2-(y_2')^2-4nx^{-1}(1+x^2)(1-x^2)^{-1}y_1'+16n(1-x^2)^{-2}(n-(n+1)(K\phi)^{-\frac{1}{n}}+K^{-\frac{1}{n}}\phi^{-\frac{n+1}{n}})=0.
\end{align}
By $(\ref{equn_expansion1})$, we have the expansion of $y_1$ and $y_2$ at $x=0$, which can also be done directly using the system $(\ref{equn_SUnEinstein01})-(\ref{equn_SUnEinstein03})$ and the boundary data $(\ref{equn_SUnBV01})$. Let $\Phi(x)$ be the function on the left hand side of the equation $(\ref{equn_SUnEinstein04})$. Take derivative of $\Phi$ and use the equations $(\ref{equn_SUnEinstein02})$ and $(\ref{equn_SUnEinstein03})$ we have
\begin{align}\label{equn_SUn2-21}
\Phi'+(y_1'-2x^{-1}(n-1+(n+1)x^2)(1-x^2)^{-1})\Phi=0.
\end{align}
Consider $y_1'$ as a given function. Using the expansion $(\ref{equn_expansion1})$, we can derive that $(\ref{equn_SUn2-21})$ has a unique solution $\Phi=0$, which is $(\ref{equn_SUnEinstein04})$. Therefore, $(\ref{equn_SUnEinstein02})$ and $(\ref{equn_SUnEinstein03})$ combining with the expansion of the Einstein metric imply $(\ref{equn_SUnEinstein04})$. Similarly, any two of the equations $(\ref{equn_SUnEinstein01})-(\ref{equn_SUnEinstein03})$ and $(\ref{equn_SUnEinstein04})$ combining with the boundary expansion of the Einstein metric give the other two equations. Notice that the coefficients of the expansion of the metric can be solved inductively by the equations $(\ref{equn_SUnEinstein02})-(\ref{equn_SUnEinstein03})$ and the initial data $(\ref{equn_SUnBV01})$ before the order $x^n$.

On $(M^{n+1}, g)$ for $n=4k+3$ with its conformal infinity $(\mathbb{S}^{4k+3}, [\hat{g}])$ where $\hat{g}$ is an $\text{Sp}(k+1)$-invariant metric, let $(r, \theta)$ be the polar coordinate chosen above. Assume $\hat{g}$ satisfies $(\ref{eqn_initialdataSpn})$ at $q$ under the local coordinate $\theta=(\theta^1,...,\theta^n)$. Denote $K=I_1I_2I_3I_4^{n-3},\,\,t_i=\frac{I_i}{I_4}$ for $1\leq i\leq 3$. Let $y_1=\log(K)$ and $y_{i+1}=\log(t_i)$ for $1\leq i\leq3$. Substituting $(\ref{eqn_metricSpn})$ and $(\ref{eqn_RicciSpn})$ to the Einstein equations $(\ref{equn_EinsteinODEs1})-(\ref{equn_EinsteinODEs3})$, we have
\begin{align}
&\label{equn_SpnEinstein01}y_1''-x^{-1}(1+3x^2)(1-x^2)^{-1}y_1'+\frac{1}{2n^2}[n(y_1')^2\,+\,((n-1)y_2'-y_3'-y_4')^2\\
&+\,(-y_2'+(n-1)y_3'-y_4')^2\,+\,(-y_2'-y_3'+(n-1)y_4')^2+(n-3)(y_2'+y_3'+y_4')^2]=0,\notag\\
&\label{equn_SpnEinstein02}y_1''-x^{-1}(2n-1+(2n+1)x^2)(1-x^2)^{-1}y_1'+\frac{1}{2}(y_1')^2\\
&+\,8(1-x^2)^{-2}[\,n(n-1)-(K^{-1}t_1t_2t_3)^{\frac{1}{n}}\,\big(\,(n-3)(n+5)-(n-3)(t_1+t_2+t_3)\notag\\
&+\frac{2(2t_1t_2+2t_1t_3+2t_2t_3-t_1^2-t_2^2-t_3^2)}{t_1t_2t_3}\,\big)\,]=0,\notag\\
&\label{equn_SpnEinstein03}y_2''-x^{-1}(n-1+(n+1)x^2)(1-x^2)^{-1}y_2'+\frac{1}{2}y_1'y_2'\\
&- 8 (1-x^2)^{-2} (K^{-1}t_1t_2t_3)^{\frac{1}{n}}[(n-1)t_1+2t_2+2t_3-n-5+\frac{2(t_1^2-(t_2-t_3)^2)}{t_1t_2t_3}]=0,\notag\\
&\label{equn_SpnEinstein04}y_3''-x^{-1}(n-1+(n+1)x^2)(1-x^2)^{-1}y_3'+\frac{1}{2}y_1'y_3'\\
&-8 (1-x^2)^{-2} (K^{-1}t_1t_2t_3)^{\frac{1}{n}}[(n-1)t_2+2t_1+2t_3-n-5+\frac{2(t_2^2-(t_1-t_3)^2)}{t_1t_2t_3}]=0,\notag\\
&\label{equn_SpnEinstein05}y_4''-x^{-1}(n-1+(n+1)x^2)(1-x^2)^{-1}y_{4}'+\frac{1}{2}y_1'y_4'\\
&-8 (1-x^2)^{-2} (K^{-1}t_1t_2t_3)^{\frac{1}{n}}[(n-1)t_3+2t_1+\,2t_2-n-5+\frac{2(t_3^2-(t_1-t_2)^2)}{t_1t_2t_3}]=0,\notag
\end{align}
for $y_i(x)\in C^{\infty}([0,1])$ ($1\leq i \leq 4$) with the boundary condition
\begin{align}\label{equn_SpnBV01}
t_i(0)=\frac{\lambda_i}{\lambda_4},\,\,K(1)=t_i(1)=1,\,\,y_j'(0)=y_j'(1)=0,
\end{align}
for $1\leq i \leq 3$ and $1\leq j \leq 4$. Combining $(\ref{equn_SpnEinstein01})$ and $(\ref{equn_SpnEinstein02})$, we have
\begin{align}
&\label{equn_SpnEinstein06}(y_1')^2-4nx^{-1}(1+x^2)(1-x^2)^{-1}y_1'
-\frac{1}{n(n-1)}[((n-1)y_2'-y_3'-y_4')^2+(-y_2'+(n-1)y_3'-y_4')^2\\
&+(-y_2'-y_3'+(n-1)y_4')^2+(n-3)(y_2'+y_3'+y_4')^2]
+\frac{16n}{n-1}(1-x^2)^{-2}\,[\,n(n-1)\notag\\
&-(K^{-1}t_1t_2t_3)^{\frac{1}{n}}\,\big(\,(n-3)(n+5)-(n-3)(t_1+t_2+t_3)+\frac{2(2t_1t_2+2t_1t_3+2t_2t_3-t_1^2-t_2^2-t_3^2)}{t_1t_2t_3}\,\big)\,]=0.\notag
\end{align}
Let $\Phi(x)$ be the function on the left hand side of the equation $(\ref{equn_SpnEinstein06})$. Take derivative of $\Phi$ and use the equations $(\ref{equn_SpnEinstein01})$ and $(\ref{equn_SpnEinstein03})-(\ref{equn_SpnEinstein05})$, we have
\begin{align}\label{equn_Spn2-21}
\Phi'+(\,\frac{1}{n}\,y_1'-4x(1-x^2)^{-1})\Phi=0.
\end{align}
Consider $y_1'$ as a given function. By the expansion of the metric at $x=0$ and the initial data, similar as the case of the
$\text{SU}(k+1)$-invariant metrics, the equation has a unique solution $\Phi=0$. Therefore, $(\ref{equn_SpnEinstein01})$ and $(\ref{equn_SpnEinstein03})-(\ref{equn_SpnEinstein05})$, combining
with the expansion of the Einstein metric at $x=0$, imply $(\ref{equn_SpnEinstein06})$.  Similarly, any four equations in the
system of the six equations $(\ref{equn_SpnEinstein01})-(\ref{equn_SpnEinstein05})$ and $(\ref{equn_SpnEinstein06})$  containing at least two of $(\ref{equn_SpnEinstein03})-(\ref{equn_SpnEinstein05})$, combining with the expansion of the Einstein metric at $x=0$, imply the other two equations.

We now turn to the uniqueness of the solution to the boundary value problem $(\ref{equn_SUnEinstein01})-(\ref{equn_SUnBV01})$. The approach for the proof of uniqueness of the solution to the boundary value problem $(\ref{equn_SUnEinstein01})-(\ref{equn_SUnBV01})$ for $n=3$ (which is the Berger sphere case), used in \cite{Li}, can be generalized directly to all odd dimensions. We only list the statement of the Lemmas required and the conclusion. For details of the proof, one is referred to \cite{Li}. By similar argument as in Section \ref{section3}, using the volume comparison and the result in \cite{LiQingShi}, we have
\begin{align*}
(\frac{Y(\mathbb{S}^n,[\hat{g}])}{Y(\mathbb{S}^n,[g^{\mathbb{S}^n}])})^{\frac{n}{2}}\leq K(0)=\lim_{x\to 0}\frac{\text{det}(\bar{h})}{\text{det}(\bar{h}^{\mathbb{H}^{n+1}}(x))}=\lim_{r\to +\infty}\frac{\text{det}(g_r)}{\text{det}(g_r^{\mathbb{H}^{n+1}}(r))}<1,
\end{align*}
with $Y(\mathbb{S}^n,[\hat{g}])$ the Yamabe constant of $(\mathbb{S}^n,[\hat{g}])$ and $g^{\mathbb{S}^n}$ the round sphere metric, where
\begin{align*}
g^{\mathbb{H}^{n+1}}=dr^2+g_r^{\mathbb{H}^{n+1}}(r)=x^{-2}(dx^2+\frac{(1-x^2)^2}{4}\bar{h}^{\mathbb{H}^{n+1}})
\end{align*}
is the hyperbolic metric.

\begin{lem}\label{lem_monotonicity01}
For the initial data $\phi(0)\neq 1$, we have $y_1'(x)>0$ for $x\in(0,1)$. Also, it holds that $y_2'(x)>0$ and $\phi(0)<\phi(x)<1$ for $x\in(0,1)$ if $\phi(0)<1$; while $y_2'(x)<0$ and $1<\phi(x)<\phi(0)$ for $x\in(0,1)$ if $\phi(0)>1$. That is to say, $K$ and $\phi$ are monotonic on $x\in(0,1)$.
\end{lem}
For the proof of Lemma \ref{lem_monotonicity01}, see Lemma 5.1 in \cite{Li}. As in Section \ref{section3}, by $(\ref{equn_SUnEinstein04})$ and the initial data $(\ref{equn_SUnBV01})$, we have
\begin{align}
y_1'&\label{equn_y1lowerorderSUn1}=2nx^{-1}(1-x^2)^{-1}[1+x^2\\
&-\sqrt{(1+x^2)^2+\frac{1}{4n^2}x^2(1-x^2)^2(y_2')^2-\frac{4}{n}x^2\,(\,n-(n+1)(\phi K)^{-\frac{1}{n}}+K^{-\frac{1}{n}}\phi^{-\frac{(n+1)}{n}})}\,\,],\notag
\end{align}
and
\begin{align}\label{ineqn_lowerboundKSUn1}
n-(n+1)(\phi K)^{-\frac{1}{n}}+K^{-\frac{1}{n}}\phi^{-\frac{(n+1)}{n}}>0,
\end{align}
for $x\in[0,1)$. This gives a lower bound of $K(0)$ under the assumption $\phi(0)>\frac{1}{(n+1)}$. And hence,
\begin{align}\label{ineqn_y1d1leftSUn1}
y_1'<2nx^{-1}(1-x^2)^{-1}[1+x^2-\sqrt{(1-x^2)^2}\,\,]=4nx(1-x^2)^{-1}.
\end{align}
We assume that the boundary value problem $(\ref{equn_SUnEinstein01})-(\ref{equn_SUnBV01})$ admits two solutions $(y_{11}, y_{12})$ and $(y_{21}, y_{22})$ with $y_{11}=\log(K_1),\,y_{12}=\log(\phi_1),\,y_{21}=\log(K_2)\,$ and $y_{22}=\log(\phi_2)$. By the same argument in Lemma 5.3 in \cite{Li}, we have
\begin{lem}\label{lem_zeroz1z2}
 Assume that $\frac{1}{n+1}<\phi(0)<n+1$ and $\phi(0)\neq 1$. For any two zeroes $0<x_1<x_2\leq 1$ of $z_1'$ so that there is no zero of $z_1'$ on the interval $x\in(x_1,x_2)$, there exists a point $x_3\in(x_1,x_2)$ so that
\begin{align}\label{inequn_signz201}
(y_{12}'+y_{22}')z_1'z_2'\big|_{x=x_3}<0.
\end{align}
Also, for any zero $0<x_2\leq 1$ of $z_1'$, there exists $\varepsilon>0$ so that for any $x_2-\varepsilon < x <x_2$, we have
\begin{align}\label{inequn_signz202}
(y_{12}'(x)+y_{22}'(x))z_1'(x)z_2'(x)>0.
\end{align}
\end{lem}
Based on Lemma \ref{lem_zeroz1z2}, we are ready to prove the uniqueness of the solution to the boundary value problem $(\ref{equn_SUnEinstein01})-(\ref{equn_SUnBV01})$ by a contradiction argument. By the same proof of Theorem 5.4 in \cite{Li}, with the integrating factor $x^{-2}(1-x^2)^3$ in $(5.13)$ in \cite{Li} replaced by $x^{1-n}(1-x^2)^n$, we have the following uniqueness theorem for the boundary value problem $(\ref{equn_SUnEinstein01})-(\ref{equn_SUnBV01})$. (Notice that Theorem 5.2 in \cite{Li} is not needed in the proof of the uniqueness of solutions to the boundary value problem there. And for $K_1(0)=K_2(0)$, by the mean value theorem, there exists a zero of $z_1'$ in $x\in(0,1)$, and Theorem 5.4 in \cite{Li} covers this case.)
\begin{thm}\label{thm_SUnmetricStability}
The solution to the boundary value problem $(\ref{equn_SUnEinstein01})-(\ref{equn_SUnBV01})\,$ for $\frac{1}{n}<\phi(0)<1$ and $1<\phi(0)<n$ must be unique if it exists.
\end{thm}
Notice that when $\phi(0)=1$, the conformal infinity is the conformal class of the round sphere, and by \cite{Andersson-Dahl}\cite{Q}\cite{DJ}\cite{LiQingShi}, the CCE manifold is isometric to the hyperbolic space. Now we are ready to prove the uniqueness result Theorem \ref{thm_someSUnmetric}.

\begin{proof}[Proof of Theorem \ref{thm_someSUnmetric}]
For the case $\phi(0)=1$ i.e., $\lambda_1=\lambda_2$ so that the conformal infinity is the round sphere, the theorem has been proved in \cite{Andersson-Dahl}\cite{Q}\cite{DJ}\cite{LiQingShi}.

Now we assume that $\lambda_1\neq \lambda_2.\,$  Pick up a point $q\in \partial M=\mathbb{S}^n$. Let $x$ be the geodesic defining function about $C\hat{g}$ with $C>0$ some constant so that $x=e^{-r}$ with $r$ the distance function on $(M,g)$ to the center of gravity $p_0\in M$, see Theorem 3.6 in \cite{Li}.  Under the polar coordinate $(x, \theta)$ with $0\leq x\leq 1$ with $\theta=0$ along the geodesic $\gamma$ connecting $q$ and $p_0$, by the discussion above we have that  the Einstein equations with prescribed conformal infinity $(\mathbb{S}^n, [\hat{g}])$ with $\hat{g}$ the homogeneous metric in $(\ref{eqn_SUnstandardmetricform})$, is equivalent to the boundary value problem $(\ref{equn_SUnEinstein01})-(\ref{equn_SUnBV01})$ along the geodesic $\gamma$ provided that the solution has non-positive sectional curvature. Then by Theorem \ref{thm_SUnmetricStability}, up to isometries, the CCE metric is unique.
\end{proof}
We now give some estimates on the solution $(y_1,\,y_2)$ with $y_1=\log(K)$ and $y_2=\log(\phi)$ to the boundary value problem $(\ref{equn_SUnEinstein01})-(\ref{equn_SUnBV01})$ based on the monotonicity lemma and $(\ref{ineqn_lowerboundKSUn1})$. The estimates are similar as that in Section \ref{section3}.
\begin{lem}\label{lem_uniformestsn01}
Let $\epsilon>0$ be any given small number. For the initial data $\frac{1}{(n+1)}+\epsilon <\phi(0)<(n+1)-\epsilon$, there exists a constant $C=C(\epsilon)>0$ independent of the solution and the initial data $\phi(0)$ such that
\begin{align}
&|y_1^{(k)}(x)|\leq C,\\
&|y_2^{(k)}(x)| \leq C,
\end{align}
with $y_i^{(k)}$ the $k-$th order derivative of $x$, for $1\leq k \leq 3$, $i=1,2$ and $x\in(0,\frac{3}{4}]$. 
\end{lem}
\begin{proof}
By the monotonicity of $y_i$ on $x\in[0,1]$ and $(\ref{ineqn_lowerboundKSUn1})$, we have
\begin{align*}
&\big(\frac{(n+1)\phi(0)-1}{n\phi^{\frac{n+1}{n}}(0)}\big)^n<K(0)\leq K(x)\leq 1,
\end{align*}
and $\phi(x)$ lies on the interval between $1$ and $\phi(0)$ for $x\in[0,1]$. By the interior estimates of the elliptic equations $(\ref{equn_SUnEinstein02})$ and $(\ref{equn_SUnEinstein03})$, there exists a constant $C=C(\epsilon)>0$ 
such that
\begin{align*}
|y_i^{(k)}(x)|\leq C\,|1-\phi(0)|,
\end{align*}
for $\frac{1}{4}\leq x \leq \frac{3}{4}$ and $i=1,2$. Now we multiply $K^{\frac{1}{2}}x^{1-n}(1-x^2)^n$ on both sides of $(\ref{equn_SUnEinstein03})$ to have
\begin{align*}
(K^{\frac{1}{2}}x^{1-n}(1-x^2)^ny_2')'+ 8(n+1)x^{1-n}(1-x^2)^{n-2}K^{\frac{1}{2}-\frac{1}{n}}\phi^{-\frac{(1+n)}{n}}(1-\phi)=0.
\end{align*}
For $x\in(0, \frac{1}{2}]$, we integrate the equation on the interval $[x,\frac{1}{2}]$ to have
\begin{align}
\label{eqn_y2singlintleft}K^{\frac{1}{2}}x^{1-n}(1-x^2)^ny_2'(x)=&\,K^{\frac{1}{2}}(\frac{1}{2})2^{n-1}(\frac{3}{4})^ny_2'(\frac{1}{2})\\
&+ 8(n+1)\int_x^{\frac{1}{2}}s^{1-n}(1-s^2)^{n-2}K^{\frac{1}{2}-\frac{1}{n}}(s)\phi^{-\frac{(1+n)}{n}}(s)(1-\phi(s))ds.\notag
\end{align}
Therefore, there exists a constant $C=C(\epsilon)>0$ independent of the solution and $\phi(0)$, such that
\begin{align*}
|y_2'(x)|\leq&\,C|1-\phi(0)|x.
\end{align*}
By $(\ref{ineqn_y1d1leftSUn1})$, for $x\in[0,\frac{1}{2}]$,
\begin{align*}
|y_1'(x)|\leq&\,\frac{4n}{3}x.
\end{align*}
Substituting these two estimates to $(\ref{equn_SUnEinstein02})$ and $(\ref{equn_SUnEinstein03})$, we have that there exists a constant $C=C(\epsilon)>0$ independent of the solution and the initial data such that
\begin{align*}
|y_i''(x)|\leq&\,C,
\end{align*}
for $x\in[0,\frac{1}{2}]$ and $i=1,2$. We take derivative of $(\ref{equn_SUnEinstein01})$ with respect to $x$ and obtain
\begin{align*}
y_1'''=&-\frac{1}{n}[y_1'y_1''+(n-1)y_2'y_2'']+ \frac{(1+3x^2)}{x(1-x^2)}\,\big(-x^{-1}y_1'+y_1''\big)+x^{-1}\frac{d}{dx}[(1+3x^2)(1-x^2)^{-1}]\,y_1'\\
=&-\frac{1}{n}[y_1'y_1''+(n-1)y_2'y_2'']+\frac{(1+3x^2)}{x(1-x^2)}[-\frac{1}{2n}\big((y_1')^2+(n-1)(y_2')^2\big)+ 4x(1-x^2)^{-1}\,y_1']\\
&+x^{-1}\frac{d}{dx}[(1+3x^2)(1-x^2)^{-1}]\,y_1'
\end{align*}
where for the second identity we have used the equation $(\ref{equn_SUnEinstein01})$ again. Therefore, by the above estimates,
there exists a constant $C=C(\epsilon)>0$ independent of the solution and $\phi(0)$ such that
\begin{align*}
|y_1'''(x)|\leq Cx,
\end{align*}
for $0\leq x \leq \frac{1}{2}$. We take derivative of $(\ref{equn_SUnEinstein03})$ with respect to $x$ and obtain
\begin{align*}
y_2'''=&[(n-1)+(1+n)x^2\,]\,x^{-1}(1-x^2)^{-1}(-x^{-1}y_2'+y_2'') + x^{-1}y_2'\,\frac{d}{dx }\big(\frac{(n-1)+(1+n)x^2}{(1-x^2)}\big) \\ &-\frac{1}{2}(y_1''y_2'+y_1'y_2'')-8(n+1)\frac{d}{dx}[(1-x^2)^{-2}K^{-\frac{1}{n}}\phi^{-\frac{1}{n}}(\phi^{-1}-1)],
\end{align*}
and therefore, by the above estimates on $y_i$ up to second order derivatives, there exists a constant $C=C(\epsilon)>0$ independent of the solution and $\phi(0)$ such that
\begin{align}\label{ineqn_y2d3SUn}
|y_2'''(x)|\leq C(x^{-1}|-x^{-1}y_2'+y_2''| + x),
\end{align}
for $0\leq x\leq \frac{1}{2}$. Now we turn to the estimate of the term $x^{-1}|-x^{-1}y_2'+y_2''|$. For $x\in[0,\frac{1}{2}]$, by $(\ref{eqn_y2singlintleft})$, we have
\begin{align}
\label{equn_y2d1leftestimatesSUn01-1}y_2'(x)=&\,(K(\frac{1}{2}))^{\frac{1}{2}}2^{n-1}(\frac{3}{4})^ny_2'(\frac{1}{2})(K(x))^{-\frac{1}{2}}x^{n-1}(1-x^2)^{-n}\\
&- 8(n+1)(K(x))^{-\frac{1}{2}}x^{n-1}(1-x^2)^{-n}\int_x^{\frac{1}{2}}s^{1-n}(1-s^2)^{n-2}(K(0))^{\frac{1}{2}-\frac{1}{n}}(\phi(0))^{-\frac{(1+n)}{n}}(1-\phi(0))ds\notag\\
&+ 8(n+1)(K(x))^{-\frac{1}{2}}x^{n-1}(1-x^2)^{-n}\int_x^{\frac{1}{2}}s^{2-n}(1-s^2)^{n-2}\,\times O(1)ds\notag\\
=&O(1)x^2 - 8(n+1)(K(0))^{\frac{1}{2}-\frac{1}{n}}(\phi(0))^{-\frac{(1+n)}{n}}(1-\phi(0))(K(x))^{-\frac{1}{2}}x^{n-1}(1-x^2)^{-n}\int_x^{\frac{1}{2}}s^{1-n}(1-s^2)^{n-2}ds \notag\\
=&O(1)x^2+\frac{8(n+1)}{n-2}(K(0))^{-\frac{1}{n}}(\phi(0))^{-\frac{(n+1)}{n}}\,(1-\phi(0)) x,\notag
\end{align}
where there exists a constant $C=C(\epsilon)>0$ independent of the solution and $\phi(0)$ so that $|O(1)|\leq C$ for all terms $O(1)$ in the formula, by the above estimates on $y_i$ and $y_i'$ for $i=1,2$. Substituting $(\ref{equn_y2d1leftestimatesSUn01-1})$ back to $(\ref{equn_SUnEinstein03})$, we have
\begin{align}
\label{equn_y2d2leftestimatesSUn01-1}y_2''=&\frac{8(n+1)(n-1)}{n-2}(K(0))^{-\frac{1}{n}}(\phi(0))^{-\frac{(n+1)}{n}}\,(1-\phi(0))+O(1)x\,-\,8(n+1)(K(0))^{-\frac{1}{n}}(\phi(0))^{-\frac{(n+1)}{n}}\,(1-\phi(0))\\
=&\frac{8(n+1)}{n-2}(K(0))^{-\frac{1}{n}}(\phi(0))^{-\frac{(n+1)}{n}}\,(1-\phi(0))+O(1)x \notag
\end{align}
where there exists a constant $C>0$ independent of the solution and $\phi(0)$ so that $|O(1)|\leq C$ for all terms $O(1)$ in the formula, by the above estimates on $y_i$ and $y_i'$ for $i=1,2$. Therefore, by $(\ref{equn_y2d1leftestimatesSUn01-1})$ and $(\ref{equn_y2d2leftestimatesSUn01-1})$, there exists a constant $C>0$ independent of the solution and $\phi(0)$ such that
\begin{align*}
x^{-1}|-x^{-1}y_2'+y_2''|\leq C,
\end{align*}
for $x\in[0,\frac{1}{2}]$, and hence by $(\ref{ineqn_y2d3SUn})$ there exists a constant $C=C(\epsilon)>0$ independent of the solution and $\phi(0)$ such that for $x\in[0,\frac{1}{2}]$,
\begin{align*}
|y_2'''(x)|\leq C.
\end{align*}
This completes the proof of the lemma.

\end{proof}

Recall that in \cite{GL}, Graham and Lee used a gauge fixing method and the Fredholm theory on certain weighted functional spaces to show that for the hyperbolic space $(M^{n+1},g)$ with its conformal infinity $(\partial M, [\hat{g}])$ (where $(\partial M, \hat{g})$ is the round sphere), there exists a CCE metric on $M$ for any given conformal infinity $(\partial M, [\hat{g}_1])$ where $\hat{g}_1$ is a small perturbation of $\hat{g}$ in $C^{2,\alpha}$ ($0<\alpha<1$) sense. Later Lee (\cite{Lee}) generalized this perturbation result to a more general class of CCE manifolds $(M^{n+1}, g)$ with a corresponding conformal infinity $(\partial M, [\hat{g}])$. In particular, the perturbation result holds for $g$ with non-positive sectional curvature, and also for $g$ which has sectional curvature bounded above by $\frac{n^2-8n}{8n-8}$ with the Yamabe constant of $\hat{g}$ non-negative.

Now we are ready to prove Theorem \ref{thm_someSUnmetric1}, the existence theorem of the boundary value problem.

\begin{proof}[Proof of Theorem \ref{thm_someSUnmetric1}]
Notice that this existence result is not a perturbation result in nature. We use continuity method. When $\lambda=1$, $\hat{g}^1$ is the round metric on the sphere $\mathbb{S}^n$, and the CCE metric $g^1$ filled in is the hyperbolic metric. For openness, let $\lambda_0 \in (\frac{1}{(n+1)}, 1]$ (resp. $\lambda_0 \in [1, n+1)$) such that $g^{\lambda_0}$ is a CCE metric on $B_1$ which is non-positively curved with $(\mathbb{S}^n, [\hat{g}^{\lambda_0}])$ as its conformal infinity. Then by \cite{Lee}, there exists $\epsilon>0$ such that for $\lambda\in (\lambda_0-\epsilon, \lambda_0+\epsilon)$, there exists a CCE metric $g^{\lambda}$ on $B_1$ with $(\mathbb{S}^n, [\hat{g}^{\lambda}])$ as its conformal infinity and the sectional curvature of $g^{\lambda}$ is close enough to that of $g^{\lambda_0}$ at the corresponding points on $B_1$ for $\epsilon>0$ small enough, and moreover, by Theorem \ref{thm_expansion1}, $g^{\lambda}$ has the smooth expansion $(\ref{equn_expansion1})$.

By \cite{Li} and the argument at the beginning of this section, if a CCE metric $g^{\lambda}$ is non-positively curved, then the Killing vector fields on $(\mathbb{S}^n, \hat{g}^{\lambda})$ are extended to $(B_1, g^{\lambda})$. And also there exists a center of gravity $p_0$ of $(B_1, g^{\lambda})$ so that $x=e^{-r}$ is a geodesic defining function with $r$ the distance function to $p_0$ on $(B_1, g^{\lambda})$. For a given point $q\in \mathbb{S}^n$ on the boundary, let $\gamma$ be the geodesic connecting $p_0$ and $q$. Let $(x, \theta)$ $(0\leq x \leq 1)$ be the local coordinate near the geodesic as defined at the beginning of this section. Then by the above argument, it has the form $(\ref{eqn_metrictwocomponents})$, $(\ref{eqn_metricSUn})$ and $(\ref{equn_changevariablesSUn})$, which satisfies the boundary value problem $(\ref{equn_SUnEinstein01})-(\ref{equn_SUnBV01})$.

For compactness, let $\{\lambda_j\}_{j>0}$ be a sequence of points converging to $\lambda_0\in (\frac{1}{(n+1)}, (n+1))$, with $g^{\lambda_j}$ a non-positively curved CCE metric with $(\mathbb{S}^n, [\hat{g}^{\lambda_j}])$ as its conformal infinity. Here $\{\hat{g}^{\lambda_j}\}_j$ are invariant under the same $\text{SU}(2k+1)$-group action on $\mathbb{S}^n$. Let $p_0^j$ be the center of gravity of $(B_1, g^{\lambda_j})$ for $j>0$. For any $(B_1, g^{\lambda_j})$, there exists a natural diffeomorphism $F_j:\,U_{p_0^j} \to \partial B_1$ between the unit tangent sphere at $p_0^j$ and $\partial B_1$, induced by the exponential map at $p_0^j$. Now for $j>1$, we define a map $H_j:\,\overline{B}_1\to \overline{B}_1$, with $H_j\big|_{\partial B_1}=\text{Id}$ and
\begin{align*}
H_j(\text{Exp}_{p_0^1}^{g^{\lambda_1}}(t\,F_1^{-1}(p)))=\text{Exp}_{p_0^j}^{g^{\lambda_j}}(t\,F_j^{-1}(p)),
 \end{align*}
 for $p\in \partial B_1$ and $t\geq 0$, where $\text{Exp}_{p_0^j}^{g^{\lambda_j}}$ is the exponential map at $p_0^j$ on $(B_1, g^{\lambda_j})$. It is easy to check that $H_j$ is a diffeomorphism. Let $g_j^{\ast}=H_j^{\ast}g_j$, which is still denoted as $g_j$, for $j\geq 2$. Now $(B_1, g_j)$ has the same center of gravity $p_0\in B_1$ for all $j>0$. Let $r$ be the distance function to $p_0$, and hence $x=e^{-r}$ is the geodesic defining function of each $j>0$.  Then for a given point $q\in \mathbb{S}^n$ on the boundary, let $\gamma$ be the geodesic connecting $q$ and $p_0$ and $(r, \theta)$ be the polar coordinate near $\gamma$ defined as above. 
  Then, under the polar coordinates $(r, \theta)$,
 \begin{align*}
 g^{\lambda_j}=dr^2+g_r^j=dr^2+ \sinh^2(r)\,\bar{h}^j=x^{-2}(dx^2+(\frac{1-x^2}{2})^2\bar{h}^j),
 \end{align*}
 and moreover these metrics have the form $(\ref{eqn_metrictwocomponents})$, $(\ref{eqn_metricSUn})$ and $(\ref{equn_changevariablesSUn})$ with $(g_r, \bar{h}, I_i(x), \phi(x), K(x))$ replaced by $(g_r^j, \bar{h}^j, I_i^j(x), \phi^j(x), K^j(x))$ and $y_i(x)$ replaced by $y_i^j(x)$ with $i=1,2$ correspondingly for any $j\geq 1$, and also these metrics satisfy the boundary value problem $(\ref{equn_SUnEinstein01})-(\ref{equn_SUnBV01})$ for $x\in[0,1]$. By Lemma \ref{lem_uniformestsn01}, there exists $C>0$ independent of $j$ such that
\begin{align*}
|(y_i^j)^{(k)}(x)|\leq C,
\end{align*}
on $x\in [0,\frac{1}{2}]$ for $i=1,2$, $0\leq k\leq 3$ and $j\geq 1$, with  $(y_i^j)^{(k)}$ the $k$-th order derivative of $y_i^j$. Therefore, up to a subsequence, $y_i^j$ converges in $C^{2,\alpha}$ norm to some function $y_i$ on $x\in[0,\frac{1}{2}]$ for any $\alpha \in (0,1)$. Moreover, by the non-positivity of the sectional curvatures of $g^{\lambda_j}$ and the Einstein equations, the norm of the Weyl tensor has the bound $|W|_{g^{\lambda_j}}\leq \sqrt{(n^2-1)n}$. The same argument in Lemma \ref{lem_yiboundrightGB} yields that there exists a constant $C>0$ independent of $j$ such that
\begin{align*}
|(y_i^j)^{(k)}(x)|\leq C,
\end{align*}
on $x\in [\frac{1}{2}, 1]$ for $i=1,2$, $0\leq k\leq 2$ and $j\geq 1$. Moreover, by the uniform bound of $|W|_{g^{\lambda_j}}$ and the Einstein equation, the sequence $g^{\lambda_j}$ are uniformly bounded in $x\in [\frac{1}{2},1]$ under any $C^k$ $(k\geq 1)$ norm. Therefore, up to a subsequence, $g^{\lambda_j}$ converges to an Einstein metric on the domain $x\in [\frac{1}{2}, 1]$ which is $r\leq \ln(2)$, in $C^k$ norm for any $k>0$ with the same Killing vector fields as $g^{\lambda_j}$. Therefore, up to a subsequence, $(B_1, p_0, g^{\lambda_j})$ converges to a CCE manifold $(B_1, p_0, g^{\lambda_0})$ in $C^{k, \alpha}$ for any $k \geq 1$ and $\alpha \in (0, 1)$ in pointed Cheeger-Gromov sense, which is non-positively curved. Also, with $x$ the same geodesic defining function of $g^{\lambda_0}$ and $g^{\lambda_j}$, $x^2g^{\lambda_j} \to x^2g^{\lambda_0}$ in $x\in[0,\frac{1}{2}]$ in $C^{2,\alpha}$ and $(\mathbb{S}^n, [\hat{g}^{\lambda_0}])$ is the conformal infinity of $(B_1, g^{\lambda_0})$. Moreover, by Theorem \ref{thm_expansion1}, the metric $g^{\lambda_0}$ is smooth and has a smooth expansion at $x=0$. This proves the compactness.

Then a direct argument of continuity method starting from $\lambda=1$ concludes the theorem.

\end{proof}

 \end{document}